\documentclass[3p,preprint]{elsarticle}
\usepackage[T1]{fontenc}
\usepackage{amsfonts}
\usepackage{amsmath}
\usepackage{amssymb}
\usepackage{color}

\newtheorem{theorem}{Theorem}[section]

\newtheorem{remark}[theorem]{Remark}

\newenvironment{proof}[1][Proof]{\noindent\textbf{#1.} }{\ \rule{0.5em}{0.5em}}
\usepackage{csquotes}

\usepackage{algorithmicx}
\usepackage{algorithm}
\usepackage{algpseudocode}
\begin{document}

\begin{frontmatter}

\title{Multinode Shepard Functions and Tensor Product Polynomial Interpolation: Applications to Digital Elevation Models
}

\author[address-DAM-Granada]{Domingo Barrera}
 \ead{dbarrera@ugr.es}

\author[address-CS]{Francesco Dell'Accio\corref{corrauthor}}
\cortext[corrauthor]{Corresponding author}
 \ead{francesco.dellaccio@unical.it}

\author[address-CS]{Filomena Di Tommaso}
 \ead{filomena.ditommaso@unical.it}

 \author[address-DAM-Granada]{Salah Eddargani}
\ead{seddargani@ugr.es}

 \author[address-DAM-Granada]{María José Ibáñez}
 \ead{mibanez@ugr.es}

\author[address-CS]{Francesco Larosa}
 \ead{francesco.larosa@unical.it}

  \author[address-CS]{Federico Nudo}
  \ead{federico.nudo@unical.it}

  \author[address-DAEGE-Granada]{Juan F. Reinoso}
  \ead{jreinoso@ugr.es}

 \address[address-DAM-Granada]{Department of Applied Mathematics and Institute of Mathematics (IMAG), University of Granada, 18071, Granada, Spain.}
  \address[address-CS]{Department of Mathematics and Computer Science, University of Calabria, Rende (CS), Italy}

 \address[address-DAEGE-Granada]{Department of Architectural and
Engineering Graphic Expression,
University of Granada, Granada, Spain}

\begin{abstract}
 The paper presents an in-depth exploration of the multinode Shepard interpolant on a regular rectangular grid, demonstrating its efficacy in reconstructing surfaces from DEM data. Additionally, we study the approximation order associated to this interpolant and present a detailed algorithm for reconstructing surfaces. Numerical tests showcase the effectiveness of the proposed algorithm. 
\end{abstract}


\end{frontmatter}

\section{Introduction}
Digital Terrain Elevation Models (DETM) serve as essential cartographic tools with different applications across civil engineering, hydrology, agriculture, environmental science, and geology. The effectiveness of these applications hinges on the methods employed for data acquisition and the subsequent transformation of this data into a DETM format. Currently, the predominant methods for data capture are photogrammetry and Light Detection and Ranging (LiDAR), prized for their expansive land coverage capability, a key feature in cartography. This data is predominantly gathered using airborne sensors, resulting in a point cloud which, upon processing, can be converted into two primary forms of DETM: an Irregular Triangle Network (TIN) or a Regular Mesh (Digital Elevation Model - DEM). A TIN is characterized by a set number of points and their corresponding three-dimensional coordinates, along with a network of adjacent points forming the edges of its triangles. Conversely, a DEM is defined by a matrix where each cell's value represents the altitude, with planimetric coordinates deducible from the cell's dimensions and the coordinates of the matrix's initial row and column.

Our focus is directed towards DEMs, a common output of various national and regional cartographic agencies. Despite representing vast territorial areas, DEMs inherently possess a discrete nature, attributed to the matrix-based format in which they are represented. Consequently, when extracting topographical features such as ground slope, data can only be ascertained at the coordinates corresponding to the centers of the matrix cells. This limitation also extends to other derivable variables from DEMs, such as orientation, curvature, or flow accumulation. These variables play a crucial role in determining drainage networks, yet their calculation is impeded by the same spatial uncertainty issues that affect slope determination \cite{Ariza-López2023}. Similar concerns have been widely documented, showing how grid resolution and interpolation choices critically affect terrain representation and hydrological modeling accuracy~\cite{zhang1994digital,chen2016robust,habib2020impact}.

The paper is organized as follows. In Section~\ref{Sec1}, we provide an in-depth description of the multinode Shepard interpolant applied on a regular rectangular grid, showcasing its utility in reconstructing surfaces from DEM data. Following this, in Section~\ref{Sec3}, we delve into the study of the approximation order associated with the multinode Shepard interpolant on this grid. Section~\ref{Sec4} outlines a detailed algorithm devised for the reconstruction of surfaces from DEM data. Finally, in Section~\ref{Sec5}, we present numerical tests, effectively illustrating the performance and efficacy of the proposed algorithm.

\section{Multinode Shepard interpolant on a rectangular grid}
\label{Sec1}
When dealing with data coming from a DEM, we assume that the model is realized by a regular grid, that is, the data are organized in a $m\times n$ matrix $\mathcal{Z}=[z_{ij}]$ of elevations, sampled at the points $(x_i,y_j)$ of the Cartesian grid
\begin{equation}
\mathcal{X}_m\times\mathcal{Y}_n=\left\{x_1,\dots,x_{m}\right\}\times\left\{y_1,\dots,y_{n}\right\},
\end{equation}
where $x_1<x_2<\dots<x_m$ and $y_1<y_2<\dots<y_n$. By assuming that 
\begin{equation}
\label{DEMvalues}
    z_{ij}=f(x_i,y_j), \qquad i=1,\dots,m,\quad j=1,\dots,n,
\end{equation}
given any point $(x,y)\in \mathcal{X}\times\mathcal{Y}$ of the minimal rectangle containing $\mathcal{X}_m\times\mathcal{Y}_n$, we want to approximate the elevation $z=f(x,y)$ with a value $\tilde{z}=\tilde{f}(x,y)$ by estimating the error $e(x,y)=z-\tilde{z}$. 

A number of techniques can be used for this goal, from the simple approach provided by the use of the bilinear polynomial, which fits a hyperbolic paraboloid to the four vertices of the grid cell interpolating linearly along the boundaries of the grid, to more elaborated and performing methods. A variety of interpolation techniques based on bivariate polynomials are assessed in \cite{kidner2003}, ranging from the bilinear polynomial to the biquintic one. General interpolation techniques, namely Inverse Distance Weighted, kringing, ANUDEM, Nearest Neighbor and Spline approaches have been compared in \cite{arun2013} (see also \cite{Ariza-López2023} for a very recent and interesting approach based on spline quasi interpolation in the Bernstein basis).  

The technique we propose aims to realize a $C^\infty$  reconstruction of the surface with polynomial reproduction properties, that guarantees higher order accuracy and fast convergence to the real surface. The key idea is to adapt the multinode Shepard methods to the current situation.

The multinode Shepard operator was recently introduced in a series of papers (see \cite{DELLACCIO2021254} and the references therein) in connection with the problem of scattered data approximation to increase polynomial precision, approximation order and accuracy of the classical Shepard operator \cite{Shepard1968}. To realize these improvements, we cover the node set 
\begin{equation*}
    \mathcal{X} = \left\{ \boldsymbol{x}_1, \dots, \boldsymbol{x}_n \right\} \subset \mathbb{R}^d
\end{equation*}
with a family of subsets $\left\{ \sigma_j \right\}_{j=1}^q \subset \mathcal{X}$, each consisting of $t = t(d,r) = \binom{d+r}{r}$ distinct neighboring nodes. Each subset $\sigma_j$ is selected so that the corresponding local interpolation problem in the polynomial space $\mathbb{P}_r\left(\mathbb{R}^d\right)$ admits a unique solution, i.e., the nodes in $\sigma_j$ are unisolvent for polynomials of total degree at most $r$. We explicitly write each subset as
\begin{equation*}
    \sigma_j = \left\{ \boldsymbol{x}_{j_1}, \dots, \boldsymbol{x}_{j_t} \right\},
\end{equation*}
where the indexes $j_\iota$ are defined through a one-to-one map $\varphi_j : \{1, \dots, t\} \rightarrow \{1, \dots, n\}$, such that $j_\iota = \varphi_j(\iota)$. This compact notation allows for a general and flexible formulation of the multinode Shepard operator. Following \cite{dell2019rate}, we define the multinode Shepard functions associated with this covering as
\begin{equation}
W_{u,j}\left( \boldsymbol{x} \right) = \frac{ \prod\limits_{\iota=1}^{t} \left\| 
\boldsymbol{x} - \boldsymbol{x}_{j_\iota} \right\|_2^{-u} }{
\sum\limits_{l=1}^{q} \prod\limits_{\lambda=1}^{t} \left\| 
\boldsymbol{x} - \boldsymbol{x}_{l_\lambda} \right\|_2^{-u}}, \quad j=1,\dots,q,\quad \boldsymbol{x} \in \mathbb{R}^d,
\label{m_points_basis_functions}
\end{equation}
where $\left\lVert \cdot \right\rVert_2$ is the standard Euclidean norm and $u > 0$ is a fixed parameter.

Let us assume that to the node set $\mathcal{X}$ is associated a set $f$ of functional data $f_1,\dots,f_n\in \mathbb{R}$ obtained by sampling a known function $f(\boldsymbol{x})$ at the points $\boldsymbol{x}_1,\dots,\boldsymbol{x}_n$, respectively, or by experimental measures. Then, by any convenient method, we can compute the polynomial $p_j[f](\boldsymbol{x})\in \mathbb{P}_r(\mathbb{R}^d)$ interpolating functional data $f(\sigma_j)=\left\{f_{j_\iota}\right\}_{\iota=1}^t$ at the nodes of $\sigma_j$ and blend these polynomials by using the multinode Shepard functions. As a result, we get the multinode Shepard operator
\begin{equation}
\mathcal{MS}_{u}[f]\left( \boldsymbol{x}\right) 
=\sum\limits_{j=1}^{q} W_{u ,j}\left( 
\boldsymbol{x}\right) p_j[f]\left( 
\boldsymbol{x}\right),\quad \boldsymbol{x}\in {\mathbb R}^d,
\label{multinode}
\end{equation}
which interpolates the function $f$ at the node set $\mathcal{X}$. Note that for $d\in\mathbb{N}$ and $q=1$ we get the classic Shepard operator. For $d=2$ and $q=3$ or $q=6$, we get the triangular Shepard \cite{dell2016IMA, cavorettofast} or the hexagonal Shepard operator \cite{dellhex2020} respectively. For $d=3$ and $q=6$ we get the tetrahedral Shepard operator \cite{Cavorettotetra}.

In presence of DEM data, the subsets $\sigma_j$ covering the node set $\mathcal{X}_m\times\mathcal{Y}_n$ can be realized easily, by exploiting the regularity of the grid and a well known interpolation technique by polynomials which is suitable for Cartesian grids, namely the tensor product interpolation \cite[Ch. 7]{CheneyLight}. In fact, arbitrary data can be interpolated uniquely by the tensor product space 
\begin{equation*}
\mathbb{P}_r(\mathbb{R})\otimes\mathbb{P}_s(\mathbb{R})=\operatorname{span}\left\{(x,y)\mapsto x^i y^j\ :\ 0\le i\le r,\ 0\le j\le s\right\}
\end{equation*}
on any subset of nodes having the form 
\begin{equation*}
\left\{x_{i_0},\dots,x_{i_r}\right\}\times\left\{y_{j_0},\dots,y_{j_s}\right\},\quad 1\le i_0<\dots<i_r\le m,\quad 1\le j_0<\dots<j_s\le n.
\end{equation*}
For this kind of polynomial interpolation explicit formulas are well known (see, f.e. \cite{CheneyLight}). Here, we will set $x_b=\displaystyle\frac{1}{r+1}\sum_{\lambda=0}^r x_{i_\lambda}$, $y_b=\displaystyle\frac{1}{s+1}\sum_{\mu=0}^s y_{j_\mu}$ and we will use the following Taylor-like representation
\begin{equation}\label{pol:tens}
p(x,y)=\sum_{\lambda=0}^r\sum_{\mu=0}^s a_{\lambda \mu} \left(\frac{x-x_b}{x_{i_\lambda}-x_b}\right)^\lambda \left(\frac{y-y_b}{y_{j_\mu}-y_b}\right)^\mu,
\end{equation}
where the coefficients $a_{\lambda \mu}$ are computed by solving, using Gaussian elimination with partial pivoting, the linear system
\[
V\boldsymbol{a}=\boldsymbol{z}
\]
of Vandermonde matrix 
\begin{equation*}
V=\left[\left(\frac{x_{i_{\alpha}}-x_b}{x_{i_{\lambda}}-x_b}\right)^{\lambda} \left(\frac{y_{j_{\beta}}-y_b}{y_{j_{\mu}}-y_b}\right)^{\mu}\right], \quad \lambda,\alpha=0,\dots,r, \quad \mu,\beta=0,\dots,s,
\end{equation*}
where $\boldsymbol{a}=\left[a_{00},\ldots,a_{r0},a_{01},\ldots,a_{r1},\ldots,a_{0s},\ldots,a_{rs}\right]^T$,  $\boldsymbol{z}=\left[z_{00},\ldots,z_{r0},z_{01},\ldots,z_{r1},\ldots,z_{0s},\ldots,z_{rs}\right]^T.$
 Exact expression for the remainder $R(f)(x,y)=f(x,y)-p(x,y)$ are known thanks to the work of D. D. Stancu. In fact, equation (4.7) of~\cite{stancu1964remainder} states that
\begin{equation}\label{remainderstancuformula}
R(f)(x,y) =\frac{u_r(x)}{(r + 1)!} f^{(r+1, 0)}(\xi, y) + \frac{v_s(y)}{(s + 1)!} f^{(0,s+1)}(x, \eta) - \frac{u_r(x)v_s(y)}{(r + 1)!(s + 1)!} f^{(r+1, s+1)}(\xi, \eta),
\end{equation}
where 
\begin{equation}\label{nodpol}
    u_r(x)=\prod_{\lambda=0}^r (x-x_{i_\lambda}), \qquad v_s(y)=\prod_{\mu=0}^s (y-y_{j_\mu}), \qquad x_{i_0}\le\xi\le x_{i_r}, \qquad y_{j_0}\le \eta\le y_{j_s},
\end{equation}
and
\begin{equation*}
 f^{(r+1, s+1)}=\frac{\partial^{r+s+2} f}{\partial x^{r+1}\partial y^{s+1}}, \qquad (r,s)\in \mathbb{N}_0\times \mathbb{N}_0. 
\end{equation*}
This equation provides a comprehensive expression for the remainder in the tensor product Lagrange interpolation, accounting for errors in both the $x$ and $y$ dimensions, as well as the interaction between these dimensions through the mixed partial derivative term. It can be used to determine a bound of the error of the multinode Shepard method in the tensorial case like that one presented in~\cite{dell2016IMA} for the case of the triangular Shepard method.

For the sake of simplicity let us assume that 
\begin{equation}\label{mod}
    \bmod(m-1,r)=\bmod(n-1,s)=0.
\end{equation}
Then the Cartesian grid $\mathcal{X}_m\times \mathcal{Y}_n$ is covered by the subsets
\begin{equation}\label{sigma_kell}
    \sigma_{k,\ell}=\{x_{(k-1)r+1},\dots,x_{kr+1}\}\times   \{y_{(\ell-1) s+1},\dots,x_{\ell s+1}\}, 
\end{equation}
$k=1,\dots,\operatorname{div}(m-1,r)$, $\ell=1,\dots,\operatorname{div}(n-1,s)$, that is
\begin{equation*}
\mathcal{X}_m\times \mathcal{Y}_n=\bigcup_{k,\ell} \sigma_{k,\ell}  
\end{equation*}
where $\operatorname{div}(\alpha,\beta)$ is the integer quotient for the division of $\alpha$ by $\beta$.  
Then, denoting by $p_{k,\ell}[\cdot](x,y)$ the polynomial~\eqref{pol:tens} based on the nodes of $\sigma_{k,\ell}$, the multinode Shepard operator on the Cartesian grid can be introduced as follows
\begin{equation}
\label{operatorMSu}
\mathcal{MS}_{u }[\cdot]\left( x,y\right) 
=\sum\limits_{k,\ell} W_{u ,k,\ell}\left( 
x,y\right) p_{k,\ell}[\cdot]\left(x,y\right),\quad (x,y)\in {\mathbb R}^2,
\end{equation}
with 
\begin{equation}
W_{u ,k,\ell}\left(x,y\right) =\dfrac{\prod\limits_{i=1, j=1}^{r+1,s+1} 
\left((x-x_{(k-1)r+i})^2+(y-y_{(\ell-1) s+j})^2\right)^{-\frac{u}{2} }}{\sum\limits_{\kappa,\lambda}
\prod\limits_{i=1, j=1}^{r+1,s+1} 
\left((x-x_{(\kappa-1) r+i})^2+(y-y_{(\lambda-1) s+j})^2\right)^{-\frac{u}{2} }}, \qquad k=1,\dots,K, \quad \ell=1,\dots, L,
\label{rectangular_basis_functions}
\end{equation}%
where 
\begin{equation}\label{KeL}
   K=\operatorname{div}(m-1,r), \qquad  L=\operatorname{div}(n-1,s).
\end{equation}
\begin{figure}[h]
    \centering
    \includegraphics[scale=0.26]{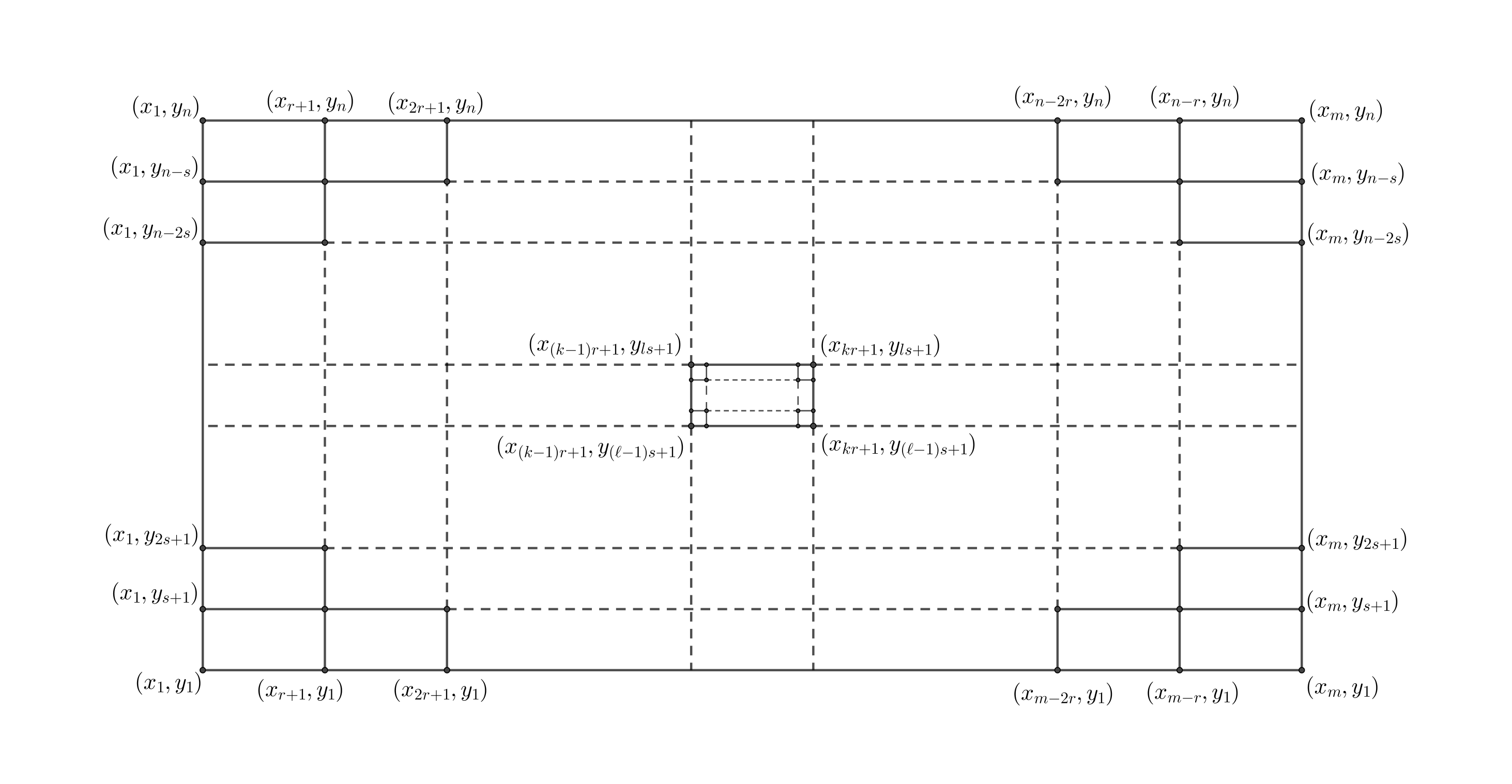}
    \caption{Grid covering of the domain $[0,1]\times[0,1]$ by overlapping rectangular subsets $\sigma_{k,\ell}$.}
    \label{grid covering}
\end{figure}
\begin{remark}
    The multinode Shepard functions defined in \eqref{rectangular_basis_functions} are non-negative and form a partition of unity~\cite{dell2019rate}, that is 
\begin{eqnarray}
   W_{u ,k,\ell}\left(x,y\right) \geq 0, \\
   \sum\limits_{k,\ell} W_{u ,k,\ell}\left(x,y\right)=1.
\end{eqnarray}
Moreover, they satisfy the following properties~\cite{dell2019rate}
\begin{eqnarray}
    &&W_{u ,k,\ell}\left(x_{\iota},y_{\tau}\right)=0,
    % &&\nabla W_{u ,k,\ell}\left(x_{\iota},y_{\tau}\right)=0, \quad u>1,\\ 
    % &&H W_{u ,k,\ell}\left(x_{\iota},y_{\tau}\right)=0, \quad u>2,
\end{eqnarray}
for any $k=1,\dots,K$, $\ell=1,\dots,L$ and $(x_{\iota},y_{\tau})\notin \sigma_{k,\ell}.$
\end{remark}
\begin{figure}
    \centering
    \includegraphics[scale=0.32]{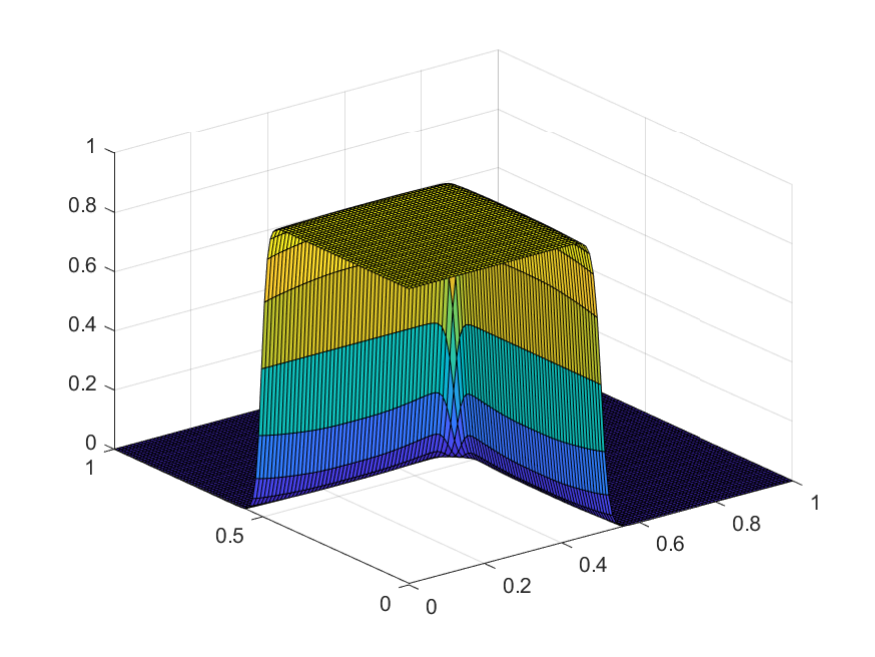}
     \includegraphics[scale=0.32]{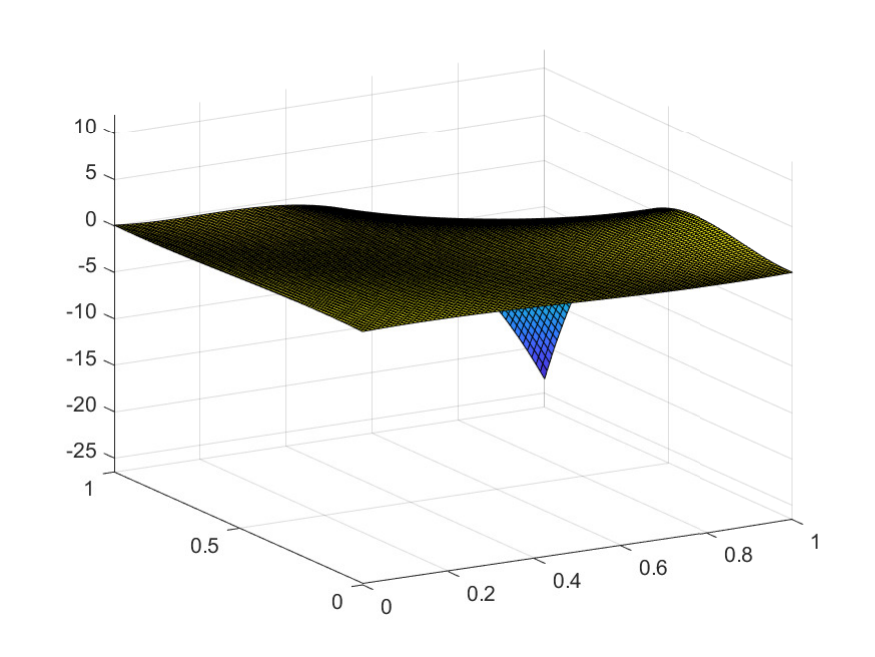}
      \includegraphics[scale=0.32]{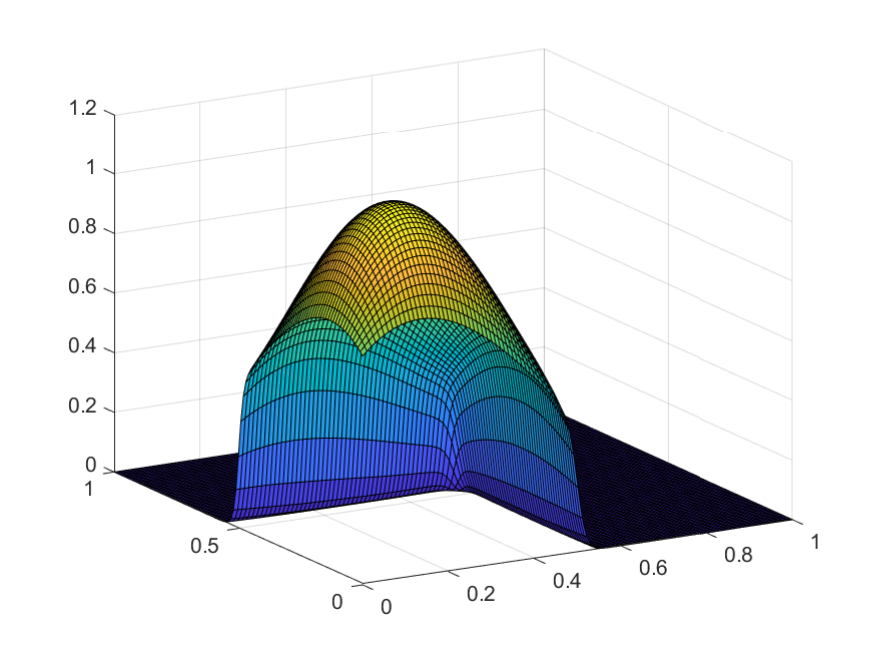}
       
   \includegraphics[scale=0.32]{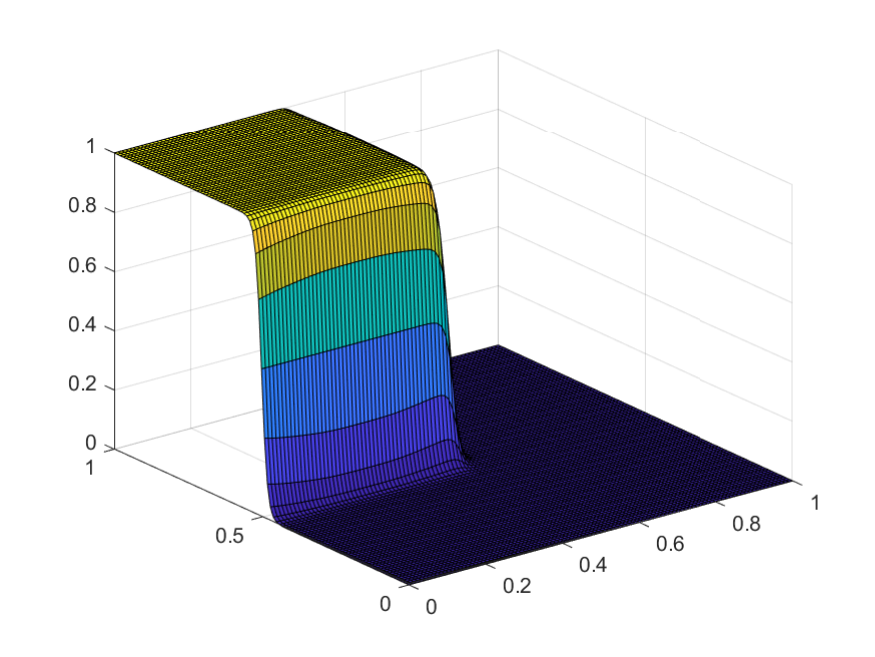}
    \includegraphics[scale=0.32]{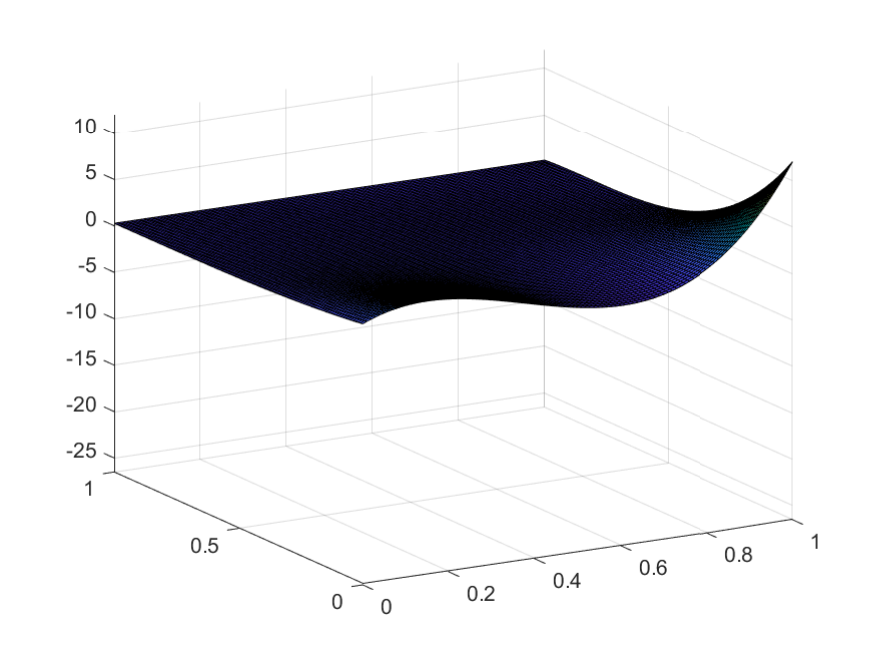}
     \includegraphics[scale=0.32]{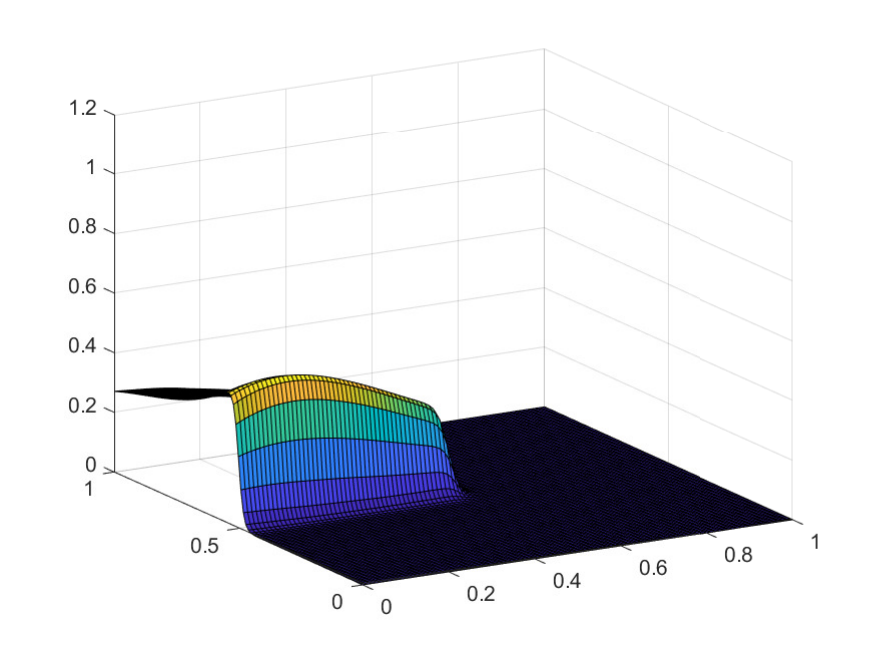}
   \includegraphics[scale=0.32]{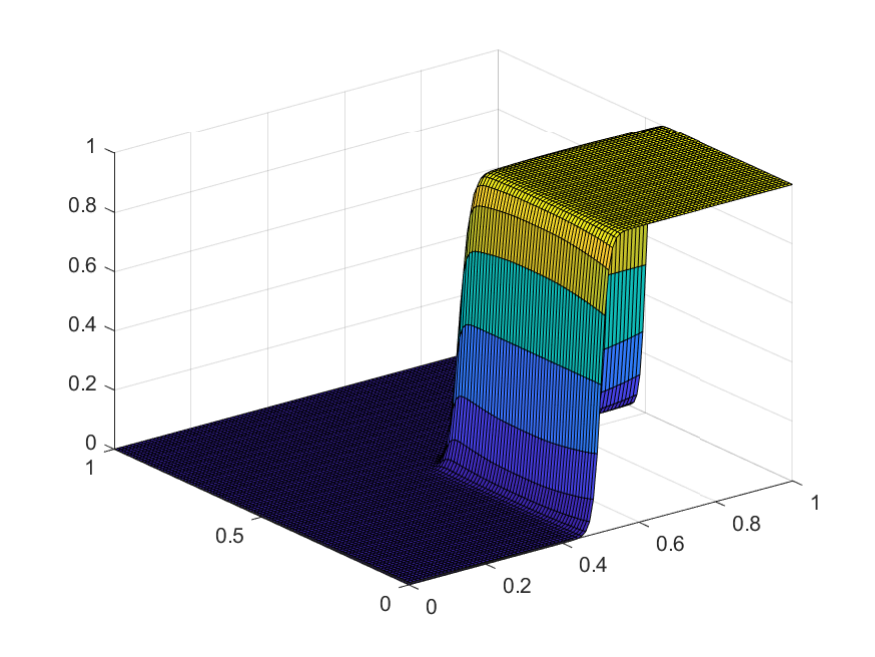}
    \includegraphics[scale=0.32]{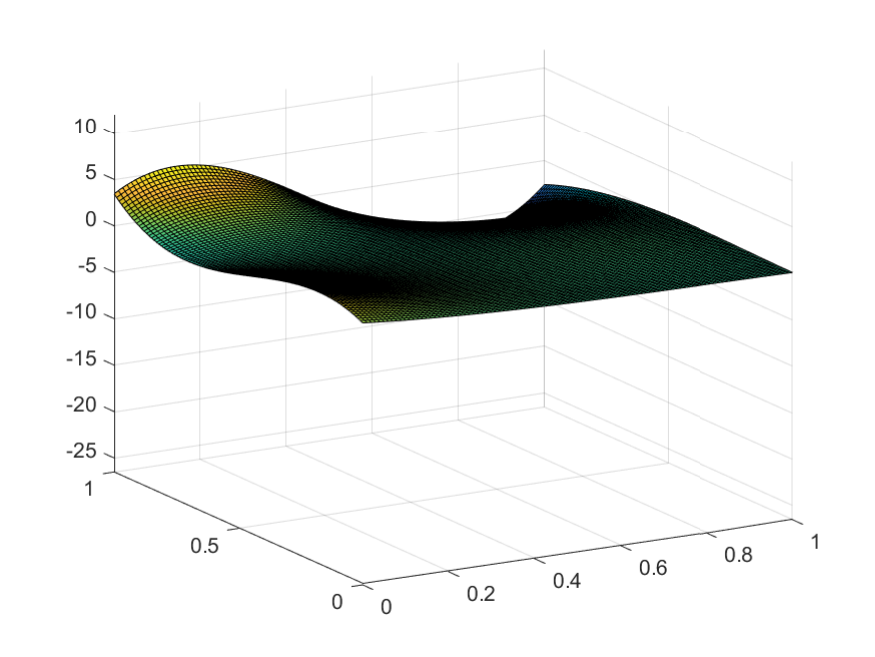}
     \includegraphics[scale=0.32]{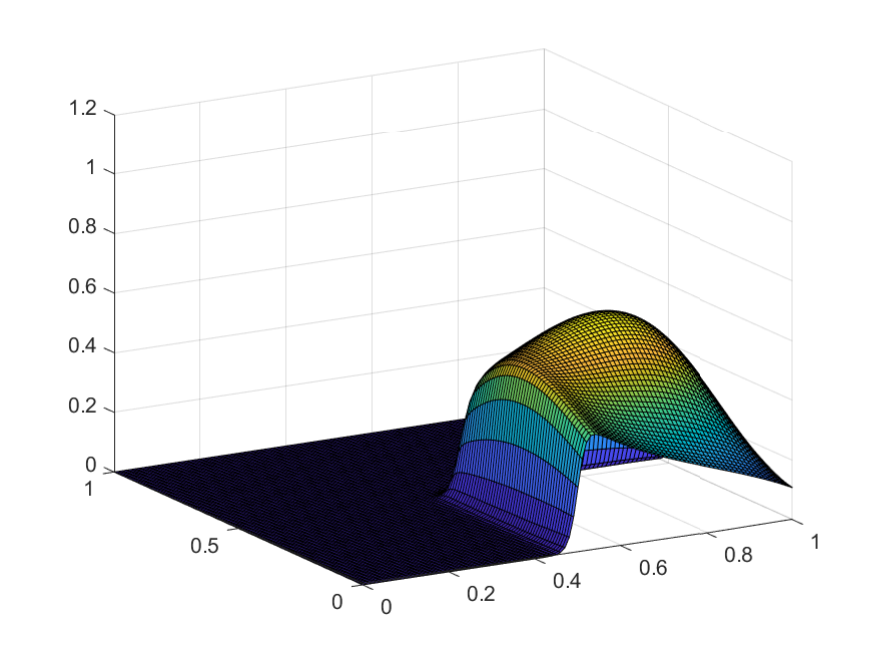}

    \includegraphics[scale=0.32]{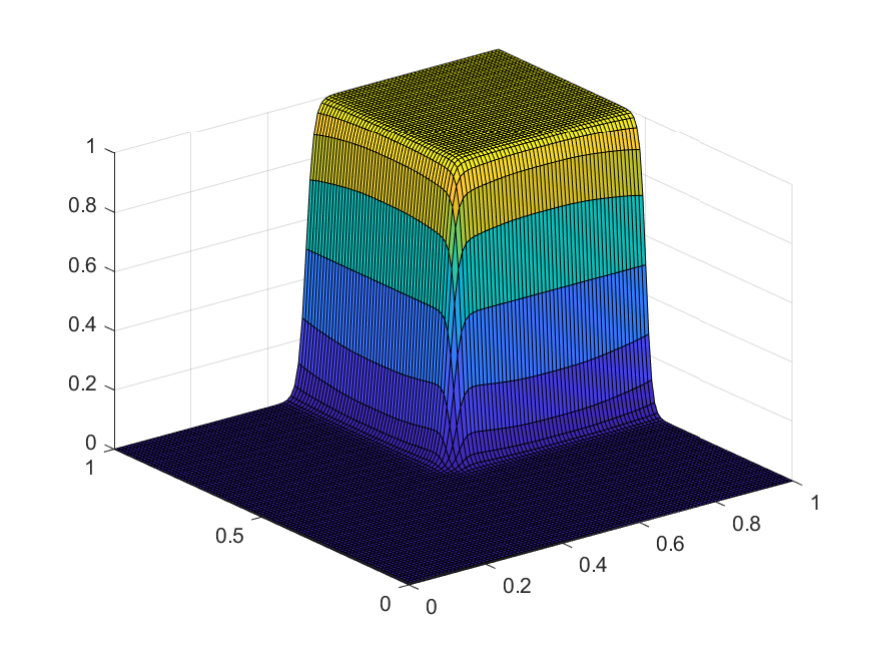}
    \includegraphics[scale=0.32]{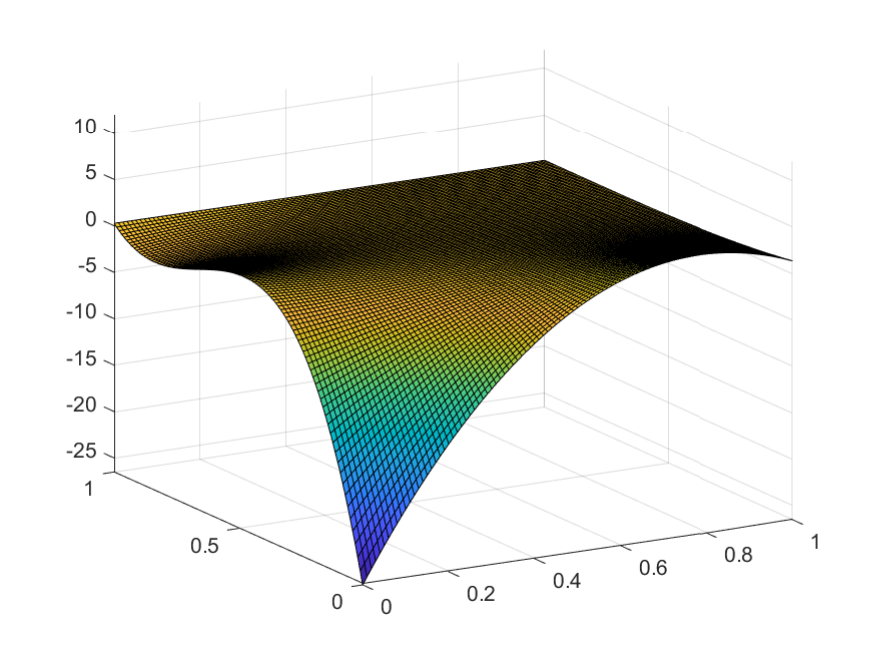}
      \includegraphics[scale=0.32]{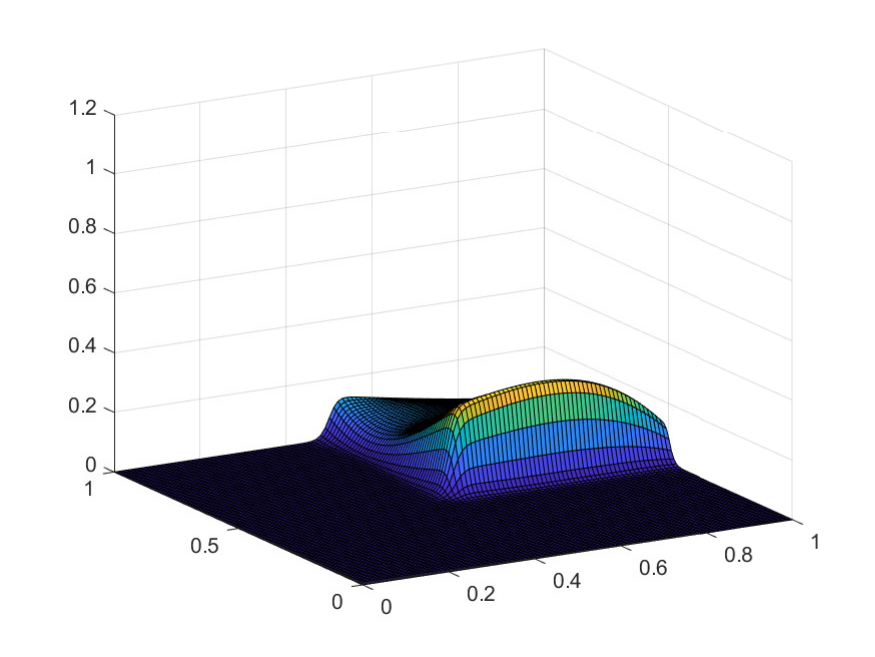}
    \caption{From the top to the bottom. Multinode Shepard basis functions $W_{2,k,\ell}$ (left), the local interpolation polynomial (center), and their pointwise product (right) used to reconstruct the Franke function from an $7\times7$ grid of uniformly spaced nodes on $[0,1]\times[0,1]$.} 
\end{figure}
\begin{figure}
    \centering
    \includegraphics[scale=0.49]{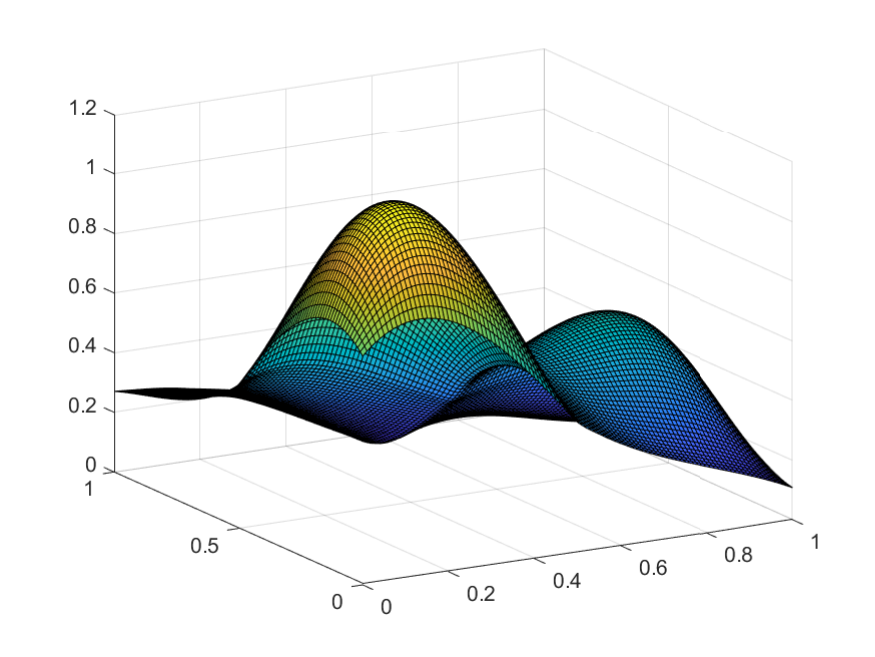}
    \caption{Reconstruction of the Franke function using the multinode Shepard operator $\mathcal{MS}_{2}$ on an $7\times7$ uniform grid over $[0,1]\times[0,1]$.} 
\end{figure}
\section{Approximation order}
\label{Sec3}
In this section, we determine the approximation order of the operator $\mathcal{MS}_{u}$, defined in~\eqref{operatorMSu}. To this aim, some preliminary settings are needed. We denote by $\Omega=\mathcal{X}\times\mathcal{Y}$ the minimal rectangle containing the node set $\mathcal{X}_m\times\mathcal{Y}_n$ and by $\left\lVert \cdot \right\rVert_{\infty}$ the maximum norm, both in the continuous and discrete cases, in particular
\begin{equation*}
    \left\lVert f \right\rVert_{\infty}=\sup_{(x,y)\in\Omega}\left\lvert f(x,y)\right\rvert, \quad f\in C(\Omega),
\end{equation*}
\begin{equation*}
    \left\lVert (x,y) \right\rVert_{\infty}=\max\left\{\left\lvert x\right\rvert, \left\lvert y\right\rvert \right\}, \quad (x,y)\in \mathbb{R}^2.
\end{equation*}
We set $J=KL$ (see~\eqref{KeL} for their definitions) and we consider the bijective map
\begin{equation}
\begin{array}{rcl}
\gamma: \{1,\dots,J\} &\rightarrow& \left\{1,\dots,K\right\} \times\left\{1,\dots,L\right\}
\\
j &\mapsto& (k,\ell)
\end{array}
\label{pilinch9C}
\end{equation}
where $k=\operatorname{div}(j-1,L)+1$, 
$\ell=\operatorname{mod}(j-1,L)+1$ are the indexes of the rectangular grid which we denote by $\sigma_j=\left\{\boldsymbol{x}_{j_{\iota}}\right\}_{\iota=1}^t$ to make use of the setting~\eqref{m_points_basis_functions} and~\eqref{multinode} for the operator~\eqref{operatorMSu}, with $q=J$, $t=(r+1)(s+1)$ and $\boldsymbol{x}=(x,y)$.
As usual, $\operatorname{div}(\alpha,\beta)$ and   $\bmod(\alpha,\beta)$ are the integer quotient and the remainder for the division of $\alpha$ by $\beta$, respectively.   
We denote by $ C^{r+1,s+1}(\Omega)$ the space of continuous functions on $\Omega$, with continuous partial derivatives up to the orders $r+1$ with respect to $x$ and $s+1$ with respect to $y$, and we set 
 \begin{equation} \label{maxfuns}
\Delta=\max\left\{\left\|f^{(r+1,0)}\right\|_{\infty},\left\|f^{(0,s+1)}\right\|_{\\infty},\left\|f^{(r+1,s+1)}\right\|_{\infty}\right\}.     
 \end{equation}
\begin{theorem}
    Let $f\in C^{r+1,s+1}(\Omega)$ and  $$u>\frac{3+r+s}{t}.$$ Then
    \begin{equation*}
        \left\|f-\mathcal{MS}_u[f]\right\|_{\infty}\leq \Delta C \left(l_{\max}\right)^{\delta_{\min}}, 
    \end{equation*}
 where $l_{\max}$ denotes the maximum of the lengths of the sides of the rectangular grid $\sigma_j$, $\delta_{\min}=\min\{r+1,s+1\}$ and $C=C(\sigma_j,u)$ is a positive constant which depends only on $\sigma_j$ and $u$. 
\end{theorem}
\begin{proof}
      If $\tilde{\boldsymbol{x}}=(\tilde{x},\tilde{y})\in \mathbb{R}^2$, we set
      \begin{equation*}
         \mathcal{Q}_{h}\left(\tilde{\boldsymbol{x}}\right)=\left\{\boldsymbol{x}\in \mathbb{R}^2\, : \, \tilde{x}-h<x<\tilde{x}+h,\quad  
         \tilde{y}-h<y<\tilde{y}+h
         \right\}.
      \end{equation*}
We further set
\begin{equation*}
    M_h= \sup_{\tilde{\boldsymbol{x}}\in\Omega}\#\left(\left\{\sigma_j\, :\,  \sigma_j \cap \mathcal{Q}_{h}({\boldsymbol{x}}) \neq \emptyset\right\}\right) 
\end{equation*}
where $\#\left(\cdot\right)$ is the cardinality operator.
 We fix $\boldsymbol{x}=({x},{y})\in\Omega$ and consider the minimal covering $\{U_{\theta}\}_{\theta=1}^{\Theta}$  
      of the set $\mathcal{X}_m\times\mathcal{Y}_n$ by square annuli (see Figure \ref{square annuli})
      \begin{equation*}
U_{\theta}=\bigcup_{\substack{\nu\in\mathbb{Z}^2 \\ \|\nu\|=\theta}} \mathcal{Q}_{l_{\max}}(\boldsymbol{x}+ l_{\max}\nu).      
\end{equation*} 
\begin{figure}[h]
    \centering
    \includegraphics[scale=0.6]{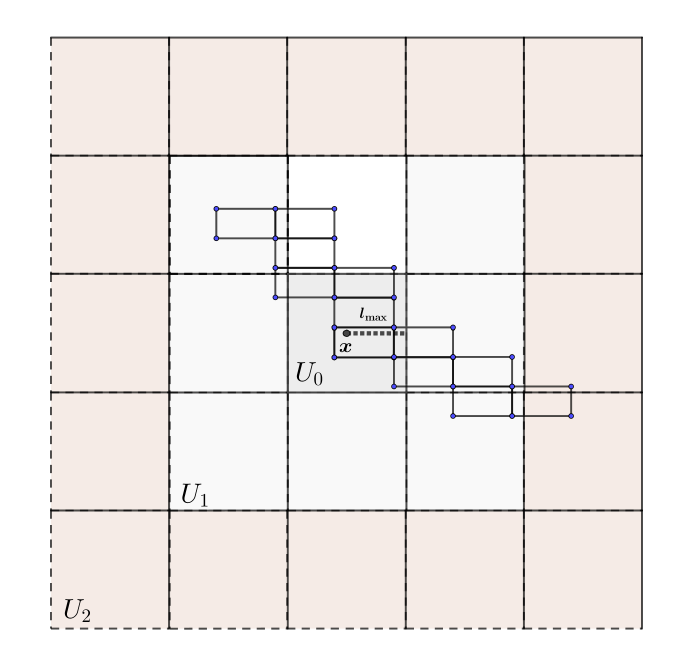}
    \caption{Relative positions of some rectangular grids $\sigma_j$ within the square annuli covering $\{U_{\ell}\}_{\ell\in \mathbb{N}}$.}
    \label{square annuli}
\end{figure}

Noting that $U_{\theta}$ is composed of $8\cdot \theta$ congruent copies of $\mathcal{Q}_{l_{\max}}({\boldsymbol{x}})$, the number of rectangular grids $\sigma_j$ with at least one vertex in $U_{\theta}$, $\theta=1,2,\dots,\Theta$ is bounded by $8 \cdot \theta \cdot M_{l_{\max}}$. 

Using~\eqref{multinode} and leveraging both the partition of unity and the non-negative properties of the multinode Shepard functions~\cite{dell2019rate}, we obtain
	\begin{equation}\label{err}
		\begin{aligned}
	e(\boldsymbol{x}):=&\left\lvert f(\boldsymbol{x})-\mathcal{MS}[f](\boldsymbol{x})\right\rvert\\
	=&\left\lvert \sum_{j=1}^J	W_{u ,j}\left( \boldsymbol{x}\right)f(\boldsymbol{x})-\sum_{j=1}^J	W_{u ,j}\left( \boldsymbol{x}\right)p_{j}(\boldsymbol{x})\right\rvert \\
	\leq&\sum_{j=1}^J\left\lvert f(\boldsymbol{x})-p_{j}(\boldsymbol{x})\right\rvert W_{u ,j}\left( \boldsymbol{x}\right).
\end{aligned}
\end{equation}
Then, by~\eqref{remainderstancuformula} and~\eqref{nodpol} and by easy computations, we get
\begin{equation}\label{err1}
	\begin{aligned}
		\left\lvert f(\boldsymbol{x})-p_{j}(\boldsymbol{x})\right\rvert&\leq \frac{\left\lvert u_r(x)\right\rvert }{(r+1)!}\left\lvert f^{(r+1,0)}(\xi,y)\right\rvert +\frac{\left\lvert v_s(y)\right\rvert}{(s+1)!}\left\lvert f^{(0,s+1)}(x,\eta)\right\rvert+\frac{\left\lvert u_r(x)\right\rvert \left\lvert v_s(y)\right\rvert}{(r+1)!(s+1)!} \left\lvert f^{(r+1,s+1)}(\xi,\eta)\right\rvert \\
		&\leq \Delta \left(\frac{\left\lvert u_r(x)\right\rvert}{(r+1)!}+\frac{\left\lvert v_s(y)\right\rvert}{(s+1)!}+\frac{\left\lvert u_r(x)\right \rvert \left\lvert v_s(y)\right\rvert}{(r+1)!(s+1)!}\right)\\
  &\leq  \Delta \left(\frac{\prod\limits_{\lambda=0}^r \left\|\boldsymbol{x}-\boldsymbol{x}_{j_\lambda}\right\|_{\infty}}{(r+1)!}+\frac{\prod\limits_{\mu=0}^s \left\|\boldsymbol{x}-\boldsymbol{x}_{j_{\mu}}\right\|_{\infty}}{(s+1)!}+\frac{\prod\limits_{\lambda=0}^r \left\|\boldsymbol{x}-\boldsymbol{x}_{j_\lambda}\right\|_{\infty}\prod\limits_{\mu=0}^s \left\|\boldsymbol{x}-\boldsymbol{x}_{j_{\mu}}\right\|_{{\infty}}}{(r+1)!(s+1)!}\right),
  \end{aligned}
\end{equation}
where $\Delta$ is defined in~\eqref{maxfuns}. 

We denote by $R_0$ the set of all rectangular grids $\sigma_j$ with at least one vertex in $U_0$. For $\theta= 1,\dots,\Theta$, we further denote by $R_{\theta}$ the set of all rectangular grids $\sigma_j$ with at least one vertex in $U_{\theta}$ and no vertex in $U_{\theta-1}$. Note that by construction,
\begin{equation*}
    \bigcup_{\theta=0}^\Theta R_{\theta}=\bigcup_{j=1}^J \sigma_{j} \quad \text{ and } \quad \bigcap_{\theta=0}^\Theta R_{\theta}=\emptyset.
\end{equation*} 
By assuming that $\boldsymbol{x}\notin\bigcup\limits_{j=1}^J \sigma_j$,  let $j^{\min}\in \{1,\dots, J\}$ such that 
\begin{equation*}
 \prod\limits_{\iota=1}^{t}\left\|\boldsymbol{x}-\boldsymbol{x}_{j^{\min}_\iota}\right\|_{{2}}= \underset{j=1,\dots,J}{\min} 	\prod\limits_{\iota=1}^{t}\left\Vert \boldsymbol{x}-\mathbf{\boldsymbol{x}}_{j_{\iota}}\right\Vert_{{2}}.   
\end{equation*}
Observe that, if $\sigma_{j} \in R_0$, we have 
\begin{equation*}
\prod\limits_{\iota=1}^{t}\left\|\boldsymbol{x}-\boldsymbol{x}_{j_\iota}\right\|_2 \leq \left(\sqrt{2}\right)^t \prod\limits_{\iota=1}^{t}\left\|\boldsymbol{x}-\boldsymbol{x}_{j_\iota}\right\|_{\infty}\leq ( 3 \sqrt{2} l_{\max})^{t},    
\end{equation*}
while if $\sigma_j \in R_{\theta}$, $\theta=1,\dots, \Theta$, then
\begin{equation*}
((2\theta-1) l_{\max})^{t} \leq \prod\limits_{\iota=1}^{t}\left\Vert \boldsymbol{x}-\mathbf{\boldsymbol{x}}_{j_{\iota}}\right\Vert_{\infty}  \leq \prod\limits_{\iota=1}^{t}\left\Vert \boldsymbol{x}-\mathbf{\boldsymbol{x}}_{j_{\iota}}\right\Vert _2\leq \left(\sqrt{2}\right)^t \prod\limits_{\iota=1}^{t}\left\Vert \boldsymbol{x}-\mathbf{\boldsymbol{x}}_{j_{\iota}}\right\Vert_{\infty} 
 \leq \left((2\theta+3)\sqrt{2}l_{\max}\right)^{t}.    
\end{equation*}

Then, we have
\begin{equation*}
		W_{u ,j}\left( \boldsymbol{x}\right)=\dfrac{\prod\limits_{\iota=1}^{t}\left\Vert 
\boldsymbol{x}-\mathbf{\boldsymbol{x}}_{j_{\iota}}\right\Vert_{{2}} ^{-u }}{\sum\limits_{l=1}^{J}%
\prod\limits_{\iota=1}^{t}\left\Vert \boldsymbol{x}-\mathbf{\boldsymbol{x}}_{l_{\iota}}\right\Vert_{2}
^{-u}}\leq \prod\limits_{\iota=1}^{t}\frac{\left\|\boldsymbol{x}-\boldsymbol{x}_{j_\iota}\right\|_{{2}}^{-u}}{\left\|\boldsymbol{x}-\boldsymbol{x}_{j^{\min}_\iota}\right\|_{{2}}^{-u}}\leq \left\{ \begin{array}{ll}
			1, & \sigma_j \in R_0,\\
			\frac{\left(3 \sqrt{2}\right)^{tu}}{(2\theta-1)^{tu}},& \sigma_j \in R_{\theta}.
		\end{array}\right.
\end{equation*}
Further assuming without loss of generality that $\boldsymbol{x}_{j_1}\in R_{\theta}$ for $ \sigma_{j}\in R_{\theta}$, so that $\left\lVert \boldsymbol{x}- \boldsymbol{x}_{j_1}\right\rVert_{{\infty}}\leq l_{\max}$ for each $\sigma_j \in R_0$ and $\left\lVert \boldsymbol{x}-\boldsymbol{x}_{j_1}\right\rVert_{{\infty}}\leq (2\theta+1)l_{\max}$ for each $\sigma_j \in R_{\theta}$, $\theta=1,\dots, \Theta$, by~\eqref{err} and~\eqref{err1} and by easy computations, it results
\begin{eqnarray*}
    	e(\boldsymbol{x}) & \leq& \Delta \left( \sum\limits_{\sigma_j \in R_{0}}\left(\frac{\prod\limits_{\lambda=0}^r \left\lVert\boldsymbol{x}-\boldsymbol{x}_{j_\lambda}\right\rVert_{{{\infty}}}}{(r+1)!}+\frac{\prod\limits_{\mu=0}^s \left\lVert\boldsymbol{x}-\boldsymbol{x}_{j_{\mu}}\right\rVert_{{{\infty}}}}{(s+1)!}+\frac{\prod\limits_{\lambda=0}^r \left\lVert\boldsymbol{x}-\boldsymbol{x}_{j_\lambda}\right\rVert_{{{\infty}}}\prod\limits_{\mu=0}^s \left\lVert\boldsymbol{x}-\boldsymbol{x}_{j_{\mu}}\right\rVert_{{{\infty}}}}{(r+1)!(s+1)!}\right)\right.\\
     &&\left.+\sum\limits_{\theta=1}^{\Theta} \sum\limits_{\sigma_j \in R_{\theta}}\left(\frac{\prod\limits_{\lambda=0}^r \left\lVert\boldsymbol{x}-\boldsymbol{x}_{j_\lambda}\right\rVert_{{{\infty}}}}{(r+1)!}+\frac{\prod\limits_{\mu=0}^s \left\lVert\boldsymbol{x}-\boldsymbol{x}_{j_{\mu}}\right\rVert_{{{\infty}}}}{(s+1)!}+\frac{\prod\limits_{\lambda=0}^r \left\lVert\boldsymbol{x}-\boldsymbol{x}_{j_\lambda}\right\rVert_{{{\infty}}}\prod\limits_{\mu=0}^s \left\lVert\boldsymbol{x}-\boldsymbol{x}_{j_{\mu}}\right\rVert_{{{\infty}}}}{(r+1)!(s+1)!}\right) 	\frac{\left(3\sqrt{2}\right)^{tu}}{(2\theta-1)^{tu}}   \right) \\
     &\leq& \Delta \left(M_h\left(\frac{l_{\max}^{r+1}}{(r+1)!}+ \frac{l_{\max}^{s+1}}{(s+1)!}+ \frac{l_{\max}^{r+s+2}}{(r+1)!(s+1)!} \right)\right.\\
     && \left.  + \sum\limits_{\theta=1}^{\Theta} \sum\limits_{\sigma_j \in R_{\theta}}\left(\frac{(2\theta+3)^{r+1}l_{\max}^{r+1}}{(r+1)!}+ \frac{(2\theta+3)^{s+1}l_{\max}^{s+1}}{(s+1)!}+\frac{(2\theta+3)^{r+s+2}l_{\max}^{r+s+2}}{(r+1)!(s+1)!} \right)	\frac{\left(3\sqrt{2}\right)^{tu}}{(2\theta-1)^{tu}} \right) \\
		& \leq & \Delta l_{\max}^{\delta_{\min}} \left(M_h\left(\frac{1}{(r+1)!}+ \frac{1}{(s+1)!}+ \frac{1}{(r+1)!(s+1)!} \right)\right.\\
  && \left.+ \sum\limits_{\theta=1}^{\Theta} 8M_h \left(\frac{(2\theta+3)^{r+1}}{(r+1)!}+ \frac{(2\theta+3)^{s+1}}{(s+1)!}+ \frac{(2\theta+3)^{r+s+2}}{(r+1)!(s+1)!} \right)	\frac{\left(3\sqrt{2}\right)^{tu}}{(2\theta-1)^{tu}} \right) \\
		& \leq & \Delta l_{\max}^{\delta_{\min}} \left(M_h\left(\frac{1}{(r+1)!}+ \frac{1}{(s+1)!}+ \frac{1}{(r+1)!(s+1)!} \right)\right.\\
  && \left.+ (8M_h)\left(3\sqrt{2}\right)^{tu} \sum\limits_{\theta=1}^{\Theta}  \left(\frac{(2\theta+3)^{r+1}}{(r+1)!}+ \frac{(2\theta+3)^{s+1}}{(s+1)!}+ \frac{(2\theta+3)^{r+s+2}}{(r+1)!(s+1)!} \right)	\frac{1}{(2\theta-1)^{tu}} \right)\\
  		& \leq & \Delta l_{\max}^{\delta_{\min}} \left(M_h\left(\frac{1}{(r+1)!}+ \frac{1}{(s+1)!}+ \frac{1}{(r+1)!(s+1)!} \right)\right.\\
  && \left.+ 8M_h\frac{\left(3\sqrt{2}\right)^{tu}}{\delta_{\min}!} \left(\sum\limits_{\theta=1}^{\Theta}\frac{(2\theta+3)^{r+1}}{(2\theta-1)^{tu}}+ \sum\limits_{\theta=1}^{\Theta}\frac{(2\theta+3)^{s+1}}{(2\theta-1)^{tu}}+ \sum\limits_{\theta=1}^{\Theta}\frac{(2\theta+3)^{r+s+2}}{(2\theta-1)^{tu}} \right)	 \right).
\end{eqnarray*}
The series $\sum\limits_{\theta=1}^{\infty}  \frac{(2\theta+3)^{r+1}}{(2\theta-1)^{tu}}$, $\sum\limits_{\theta=1}^{\infty} \frac{(2\theta+3)^{s+1}}{(2\theta-1)^{tu}}$ and $\sum\limits_{\theta=1}^{\infty} \frac{(2\theta+3)^{r+s+2}}{(2\theta-1)^{tu}}$ converge for $u>\frac{3+r+s}{t}$, then we conclude that the approximation order of $\mathcal{MS}_u[f](\boldsymbol{x})$ is $(l_{\max})^{\delta_{\min}}$.
\end{proof}

\section{Algorithm}
\label{Sec4}
From now on, we assume that the points $(x_{\mu},y_{\nu}) \in \mathcal{X}_m\times\mathcal{Y}_n$ are collected in two matrices $X, Y\in \mathbb{R}^{n\times m}$ such that $X(\nu,\mu)=x_{\mu}$ and $Y(\nu,\mu)=y_{\nu}$, $\nu=1,\dots,n,$ $\mu=1,\dots,m$. In correspondence, the elevation $z_{\mu \nu}=f(x_{\mu},y_{\nu})$ are collected in a matrix $Z \in \mathbb{R}^{n\times m}$ such that $Z(\nu,\mu)=z_{\mu \nu}$. Starting from the matrices $X, \ Y$ and $Z$ we introduce the vectors $\boldsymbol{x}, \ \boldsymbol{y}$ and $\boldsymbol{z}$ as follows 
\begin{eqnarray*}
X(\nu,\mu)&\rightarrow& \boldsymbol{x}((\mu-1)n+\nu), \\
Y(\nu,\mu)&\rightarrow& \boldsymbol{y}((\mu-1)n+\nu), \\
Z(\nu,\mu)&\rightarrow& \boldsymbol{z}((\mu-1)n+\nu),
\end{eqnarray*}
$\nu=1,\dots,n$, $\mu=1,\dots,m$. After this setting the point $(x_{(k-1)r+1},y_{(\ell-1) s+1})$ which constitutes the left bottom vertex of the subset $\sigma_{k,\ell}$ defined in \eqref{sigma_kell} (see figure \ref{grid covering}) corresponds to the index 
\[
i(k,\ell)=(k-1)rn+(\ell-1)s+1, \quad k=1,\dots,K, \ \ell=1,\dots,L.
\]
Therefore, the nodes of $\sigma_{k,\ell}$, are
\[(x_{(k-1)r+i},y_{(\ell-1)s+j})=(\boldsymbol{x}(i(k,\ell)+(i-1)n+j-1)),\boldsymbol{y}(i(k,\ell)+(i-1)n+j-1)), \]
$k=1,\dots,K$, $\ell=1,\dots,L$, $i=1,\dots,r+1$ and $j=1,\dots,s+1$. We denote by $(x_{k,\ell},y_{k,\ell})$ the barycenter of the subset $\sigma_{k,\ell}$, by
$\boldsymbol{a}_{k,\ell}$ the vector of coefficients of polynomial $p_{k,\ell}=p_{k,\ell}(x,y)$ defined in \eqref{pol:tens}, based on the nodes of $\sigma_{k,\ell}$ and expressed as in~\eqref{pol:tens}.

\begin{algorithm}
\caption{Multinode Shepard (for rectangular grids)}\label{alg:tolerance}
\begin{algorithmic}
\Require $X$, $Y$, $Z$, $r$, $s$, $u$, $x$, $y$
\Ensure $\mathcal{MS}_{u}$
\State  $[m,n]\gets \text{size}(X)$
\State  $K\gets \operatorname{div}(m-1,r)$
\State $L\gets \operatorname{div}(n-1,s)$
\For{$\mu=1, \dots, m$}
\For{$\nu=1, \dots, n$}
\State $\boldsymbol{x}((\mu-1)n+\nu)\gets X(\nu,\mu)$
\State $\boldsymbol{y}((\mu-1)n+\nu)\gets Y(\nu,\mu)$
\State $\boldsymbol{z}((\mu-1)n+\nu)\gets Z(\nu,\mu)$
\State $D((\mu-1)n+\nu)\gets\left((x-\boldsymbol{x}((\mu-1)n+\nu))^2+(y-\boldsymbol{y}((\mu-1)n+\nu))^2\right)^{-\frac{u}{2}}$
\EndFor
\EndFor
\State $Num\gets 0$
\State $Den\gets 0$
\For{$k=1, \dots, K$}
\For{$\ell=1, \dots, L$}
\State $d_{k,\ell}\gets \prod\limits_{i=1, j=1}^{r+1,s+1} D(i(k,\ell)+(i-1)n+j-1)$
\State Compute $(x_{k,\ell},y_{k,\ell})$
\State Compute $\boldsymbol{a}_{k,\ell}$
\State Compute $p_{k,\ell}$
\State $Num\gets Num+d_{k,\ell} \cdot p_{k,\ell}$
\State $Den\gets Den+d_{k,\ell}$
\EndFor
\EndFor
\State
$\mathcal{MS}_{u}\gets {Num}/{Den}$
\end{algorithmic}
\end{algorithm}

\section{Use case analysis}
\label{Sec5}
To study the performance of the method developed in this study, the algorithm has been applied to a real digital terrain model (DTM for short) defined as a regular grid. The geographical area is located in a mountainous area of Granada (South of Spain) within the Penibetic mountain range of Sierra Nevada (Figure~\ref{OrtoimagenPendientes} (left)).

\begin{figure}[h!]
\centering
\includegraphics[scale=0.25]{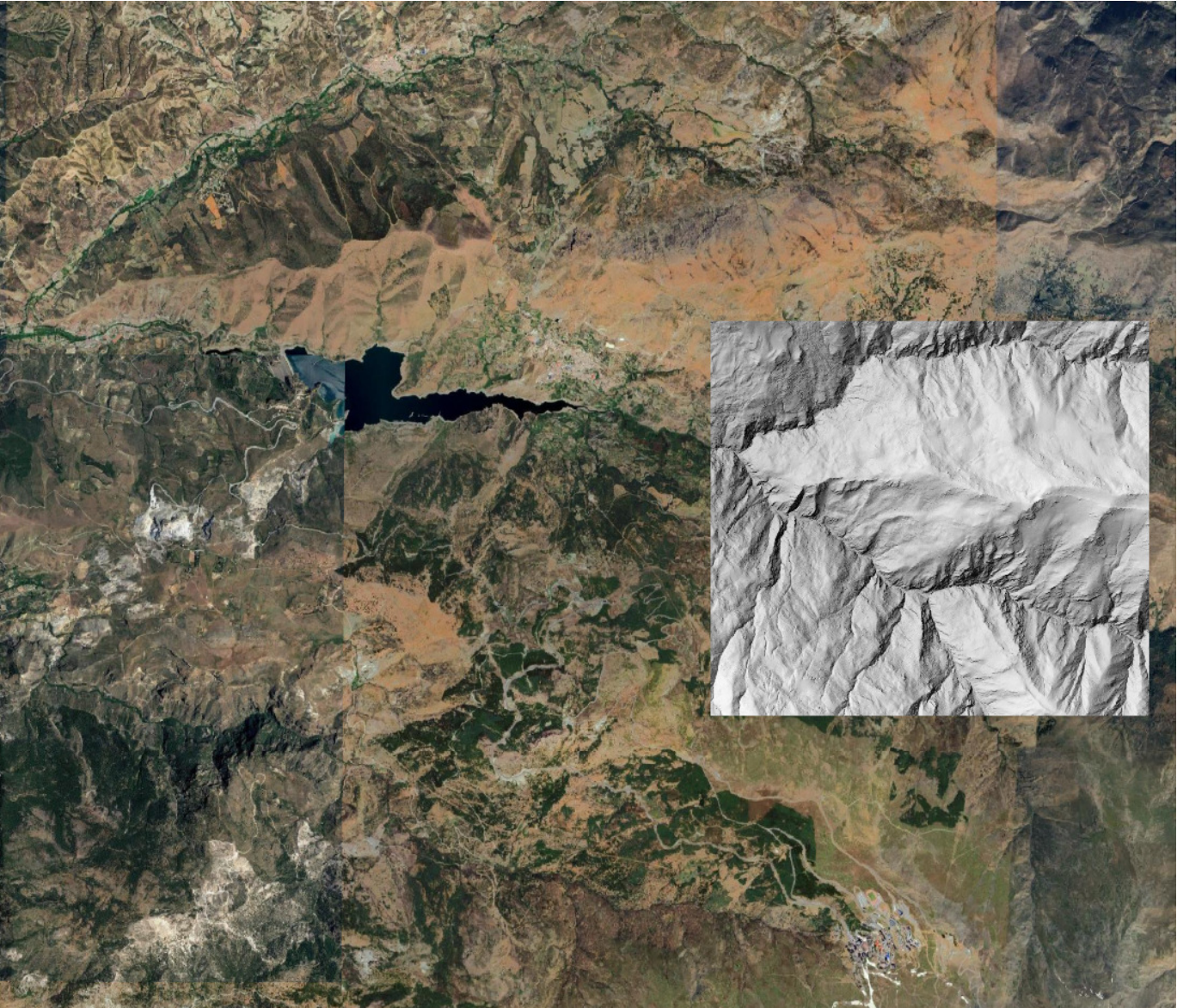}
\includegraphics[scale=0.3]{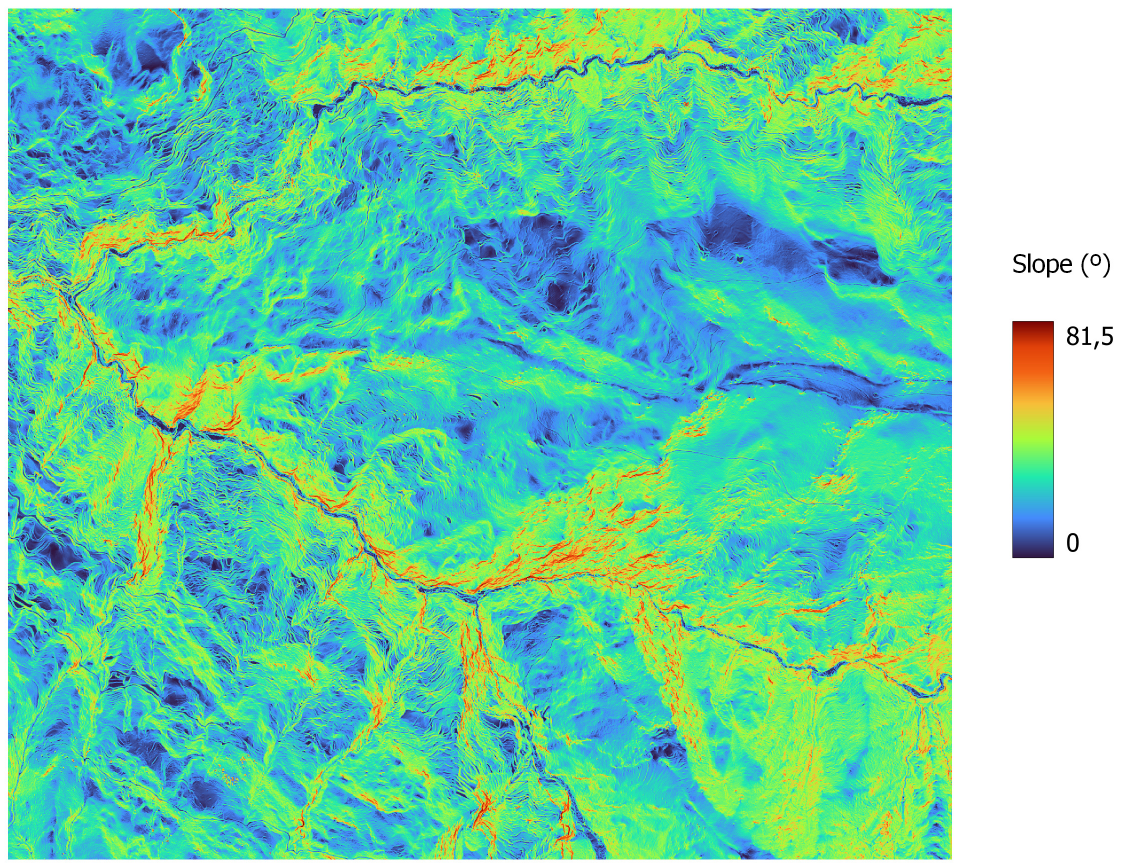}
\caption{Orthophoto of the region including the area of interest and its slope map.}%
\label{OrtoimagenPendientes}%
\end{figure}

%\begin{figure}[h!]
 %   \centering
  %  \includegraphics[width=0.5\linewidth] {Ortoimagen.eps}
   % \caption{Landscape study area.}
    %\label{Ortoimagen}
%\end{figure}

The reference DTM, DTM$_{\text{ref }2\times2}$, has been developed by the Instituto Geográfico Nacional (IGN) as the Spanish Cartographic Agency and consists of a $2\operatorname{m}$ resolution raster model. The extent of the study area is $5420
\operatorname{m}\times 4886\operatorname{m}$ with a difference in altitude of $1055 \operatorname{m}$. It is a mountainous area with an average slope equal to $28\%$ and minimum and maximum values equal to $0\%$ and $81.5\%$, respectively (Figure~\ref{OrtoimagenPendientes} (right)).

%\begin{figure}[h!]
 %   \centering
  %  \includegraphics[width=0.5\linewidth]{Fig 2 Mapa_pendientes_leyenda.png}
   % \caption{Slope map ranging from $0^\circ$ to $81.5^\circ$.}
   % \label{Mapa_pendientes_leyenda}
%\end{figure}

Figure~\ref{texturizado y sombras2} provides a 3D perspective view of the DTM to better appreciate its topography, as well as its hill-shade representation.
\begin{figure}[h!]
\centering
\includegraphics[scale=0.17]{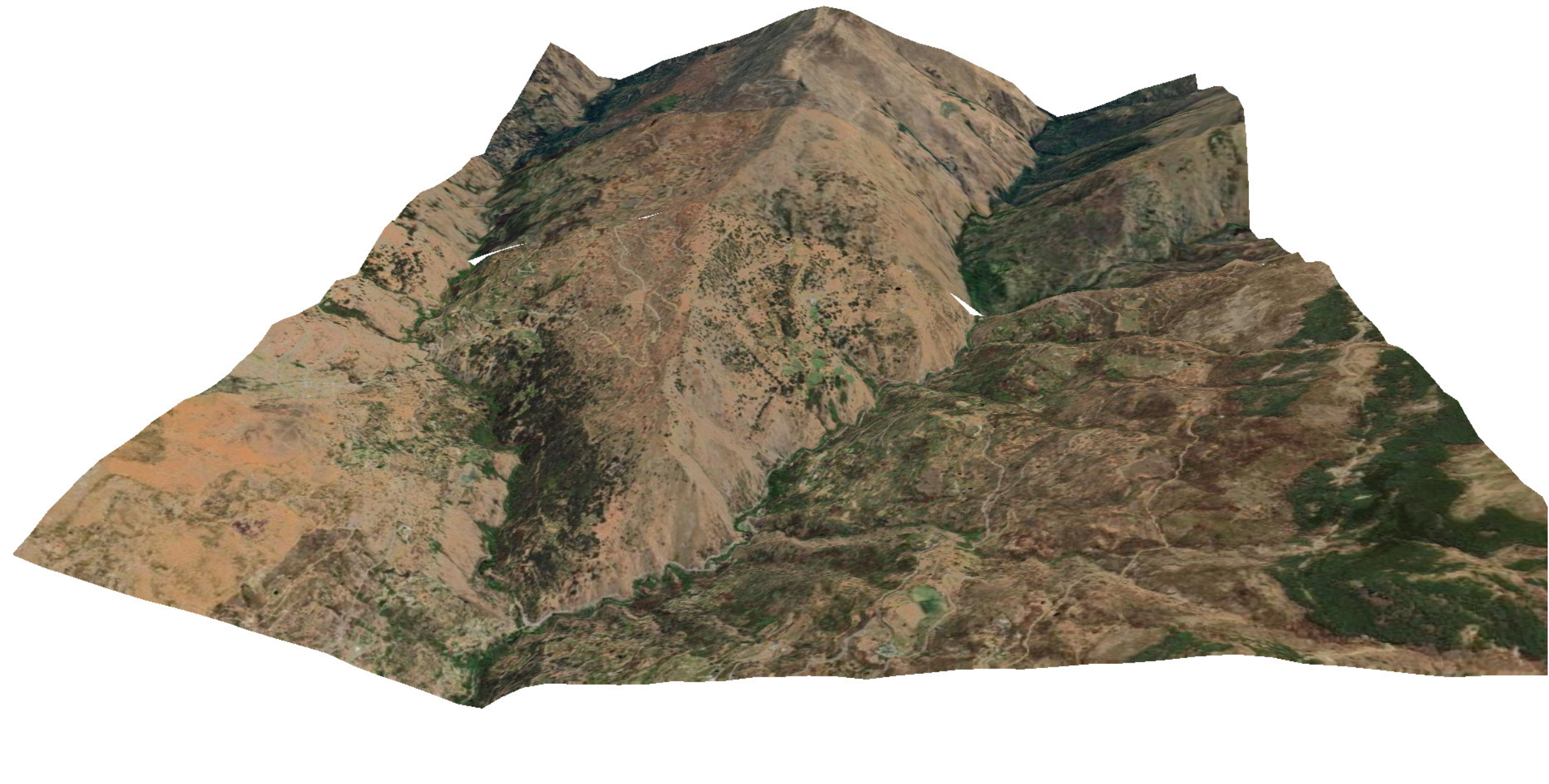}
\includegraphics[scale=0.17]{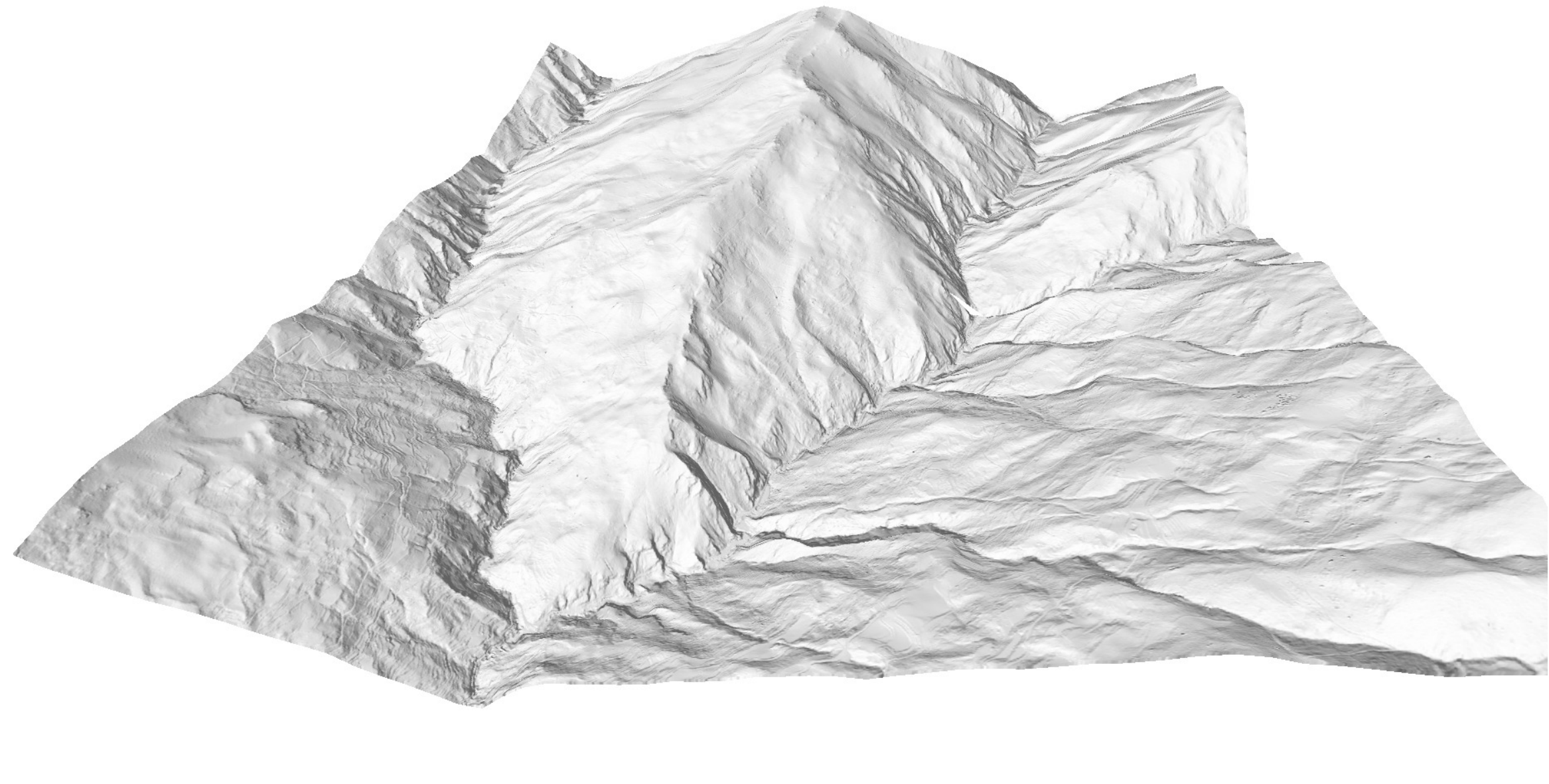}
\caption{On the left, 3D textured image corresponding to DTM$_{\text{ref
}2\times2}$. On the right, its hill-shade representation.}%
\label{texturizado y sombras2}%
\end{figure} 

%\begin{figure}[h!]
% \centering
% \includegraphics[width=0.5\linewidth]{Fig3 Perspectiva 3D texturizado y sombras2.png}
% \caption{Study area: 3D textured perspective %(up) and hillshade (bottom).}
% \label{texturizado y sombras2}
%\end{figure}
%The approximation method developed here has been compared with the cubic B-spline method commonly used in resampling operations in geographic information systems (GIS) environments \cite{HsiehAndrews}. 

The approximation method developed in this study was compared with the cubic B-spline resampling technique, a standard approach implemented in the open-source Quantum Geographic Information System (QGIS) \cite{qgis}. QGIS is a powerful open source platform extensively utilized in DEM analysis. It allows for detailed terrain assessment, including slope and aspect calculation, hydrological modeling and 3D visualization. QGIS provides support for a wide range of DEM datasets, including SRTM and LiDAR, and offers tools for interpolation, hill shading and watershed delineation, for more details see \cite{qgis_paper}. The QGIS software provides spline interpolation techniques through the GRASS GIS plugin, namely through the v.surf.bspline module \cite{grass}, which implements the cubic B-spline interpolation with Tikhonov regularisation to guarantee numerical stability. The interpolation procedure consists of two main steps: first, the linear coefficients of a spline function - made up of piecewise polynomials joined with smoothness constraints - are estimated from the input data points by least-squares regression; second, the interpolated surface is computed from these coefficients. The choice of the spline knots is crucial for optimal performance; in the case of regularly distributed data points, setting the spline step length equal to the maximum distance between points in the x- and y-directions generally gives satisfactory results. For further details, see \cite{grass_paper}.

% [AGGIUNGERE H. Hsieh, H. Andrews, Cubic splines for image interpolation and digital filtering", IEEE Transactions on Acoustics, Speech, and Signal Processing 26(6) (1978) 508--517 ALLA BIBLIOGRAFIA, label HsiehAndrews] 
Accuracy analyzes were performed independently on the horizontal and vertical components of the DTM. This accuracy assessment methodology has already been successfully applied by Ariza-López et al. 2022 (see~\cite{Ariza-López2023}) using the bicubic method provided by QGIS.
The performance analysis of the proposed method follows the following workflow:
\begin{enumerate}
    \item The original DTM (DTM$_{\text{ref }2\times2}$) of $2\operatorname{m}$ resolution is downscaled to a cell resolution of $8 \operatorname{m}$ and $16\operatorname{m}$ respectively, resulting in two downscaled DTMs of $8\operatorname{m}$ and $16\operatorname{m}$ (DTM$_{\text{up\_8m}}$ and DTM$_{\text{up\_16m}}$). The reduction method is called upscaling.
    \item The reduction method is called upscaling. For each of these two DTMs we have applied the biquadratic multinode Shepard operator with parameter $\mu=4$ to construct the corresponding interpolants, resulting in two surfaces (S$_{\text{up\_8m}}$ and S$_{\text{up\_16m}}$) whose elevations can be evaluated over the entire DTM$_{\text{ref }2\times2}$ domain, particularly at the points of the mesh multiples of 2, which are the original equispaced $(x,y,z)$ data.
    \item  S$_{\text{up\_8m}}$ and S$_{\text{up\_16m}}$ are evaluated at the points whose coordinates $(x,y,z)$ coincide with those of the $\text{DTM}_{ori\_2m}$  mesh giving rise to DTM$_{\text{dow\_Sh\_8m\_to\_2m}}$ and DTM$_{\text{dow\_Sh\_16m\_to\_2m}}$, which would be similar to a downscaling process in the field of raster resampling. In this way, the  DTM$_{\text{ref }2\times2}$  coordinates can be directly compared with those of DTM$_{\text{dow\_Sh\_8m\_to\_2m}}$ and DTM$_{\text{dow\_Sh\_16m\_to\_2m}}$ to determine the quality of the approximation provided by the proposed method. This comparison will provide a measure of the vertical accuracy of the method used.
    \item The analysis of horizontal accuracy is based on the following idea. Let us imagine that a reference DTM (Figure~\ref{explicacion}(a) and an approximate DTM (Figure~\ref{explicacion}(b)) are given.
    \begin{figure}[h!]
    \centering
    \includegraphics[scale=0.5]{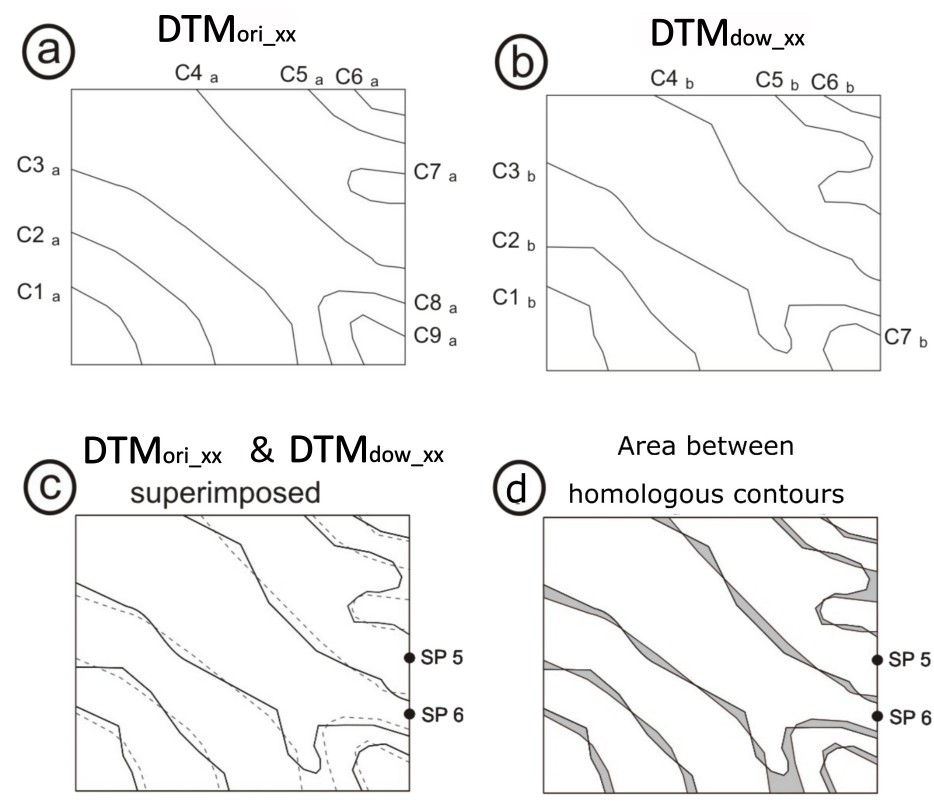} 
    \caption{Illustrating how to measure the horizontal accuracy of an approximate DTM with respect to a reference DTM.}%
    \label{explicacion}%
    \end{figure}
    Their sets of contours are available. After superimposing both figures, the algorithm proposed in \cite{Reinoso} 
    % [AGGIUNGERE J.F. Reinoso, An algorithm for automatically computing the horizontal shift between homologous contours from DTMs, ISPRS Journal of Photogrammetry and Remote Sensing 66(3) (2011) 272--286 ALLA BIBLIOGRAFIA, label Reinoso]
    is used to determine the homologous contours, calculate the area of the regions determined by them and compute the ratio between that area and the average length of the homologous contours. Therefore, all that needs to be done is to calculate the same set of contour lines of user-defined elevation in both DTM$_{\text{ref }2\times2}$ and its counterparts DTM$_{\text{dow\_Sh\_8m\_to\_2m}}$ and DTM$_{\text{dow\_Sh\_16m\_to\_2m}}$,
   resulting in comparable contour maps CL$_{\text{ref }2\times2}$,  CL$_{\text{dow\_Sh\_8m\_to\_2m}}$ and CL$_{\text{dow\_Sh\_16m\_to\_2m}}$. The horizontal discrepancy ($H_d$) will be calculated as the total area between homologous contour lines ($\sum A_{H_c}$) divided by the average length of all contour lines ($\sum C_{ref} + \sum C_{dow}) / 2$:
\[
H_d = \frac{\sum A_{H_c}}{\left(\frac{\sum C_{ref} + \sum C_{dow}}{2}\right)}.
\]
As an example, Figure~\ref{Curvas3Dperspectiva} shows an orthogonal view of homologous contours associated with CL$_{\text{ref }2\times2}$ and CL$_{\text{dow\_Sh\_16m\_to\_2m}}$ as well as an enlarged detail showing in green the area between homologous contours illustrating the horizontal discrepancy between the DTM$_{\text{ref }2\times2}$ and DTM$_{\text{dow\_Sh\_16m\_to\_2m}}$. A 3D perspective of the detail area is also shown.
\begin{figure}
    \centering
    \includegraphics[width=0.5\linewidth]{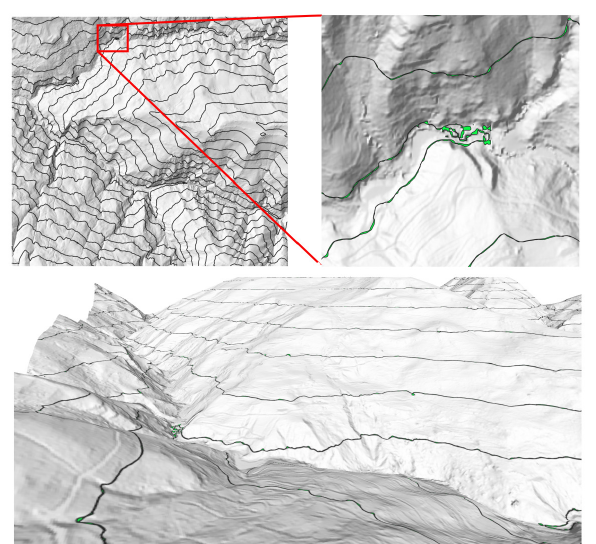}
    %%% Ho aggiunto la figura Curvas 3D perspectiva y ortogonal.eps, ma dà un problema di compilazione. La figura png resta.
    \caption{Horizontal discrepancy between homologous contours in CL$_{\text{ref }2\times2}$ and CL$_{\text{dow\_cBspl\_16m\_to\_2m}}$. The bottom image shows a 3D perspective view.}
    \label{Curvas3Dperspectiva}
\end{figure}
\item In order to compare the results of the interpolation method with other consolidated methods, steps 2 to 5 have been repeated using the above cubic B-spline method, to produce the contour line maps
CL$_{\text{dow\_cBspl\_8m\_to\_2m}}$ and
CL$_{\text{dow\_cBspl\_16m\_to\_2m}}$, which will be compared with the reference CL$_{\text{ref }2\times2}$.
\end{enumerate}
\section*{Results}
There is no visual difference between the
hill-shade figure of the original DTM and those of the surfaces S$_{\text{up\_8m}}$ and S$_{\text{up\_16m}}$ constructed from a much smaller number of points, even though the spacing is much wider.

The differences between the altitudes of the original DTM and those of the DTMs calculated from the multinode Shepard method and those provided by the QGIS cubic B-spline method have been calculated. To avoid trade-offs between positive and negative accuracy values, all vertical discrepancies have been taken in absolute value. Table~\ref{Tab1} shows the mean and standard deviation of the results.

%We note that very good results have been obtained despite having used DTMs with large step sizes to construct the surfaces representing the terrain using the biquadratic multinode Shepard method, 

 It is worth noting that the biquadratic multinode Shepard method produces very accurate results even with widely spaced input data. In our application, it effectively reconstructs high-resolution terrain features 2m apart from coarsely sampled elevation data spaced 16m apart in the x- and y- directions, highlighting its suitability for DEM downscaling applications.

\begin{table}[tbp]
\centering%
\begin{tabular}
[c]{|c|c|c|c|c|}\hline
Vertical & \multicolumn{2}{|c|}{Multinode Shepard} &
\multicolumn{2}{|c|}{{Cubic B-spline (QGIS)}}\\\cline{2-5}%
Accuracy & Mean ($\operatorname{m}$) & Sd ($\operatorname{m}$) & Mean
($\operatorname{m}$) & Sd ($\operatorname{m}$)\\\hline
DTM$_{\text{dow\_Sh\_8m\_to\_2m}}$ & $0.277$ & $0.420$ & $0.596$ &
$0.606$\\\hline
DTM$_{\text{dow\_Sh\_16m\_to\_2m}}$ & $0.549$ & $0.789$ & $0.886$ &
$1.018$\\\hline
\end{tabular}
\caption{Vertical accuracy measured as discrepancy comparing the biquadratic multinode Shepard method and the cubic B-spline method in QGIS.}%
\label{Tab1}%
\end{table}

Figure~\ref{manchas} (top left) shows in the orthogonal projection of the hill-shade representation associated with S$_{\text{up\_16m}}$ the points where the error is greater than $3$ meters after applying the biquadratic multinode Shepard operator and the cubic B-spline method (in red the errors due to QGIS cubic B-spline interpolation and in green those corresponding to the Shepard method). As expected, they occur in areas of steep slope. 3D perspectives are also presented to better appreciate the differences and overlaps, as well as the fact that these discrepancies occur in rough areas with steep slopes. 

\begin{figure}[tbp]
\centering
\includegraphics[scale=0.2]{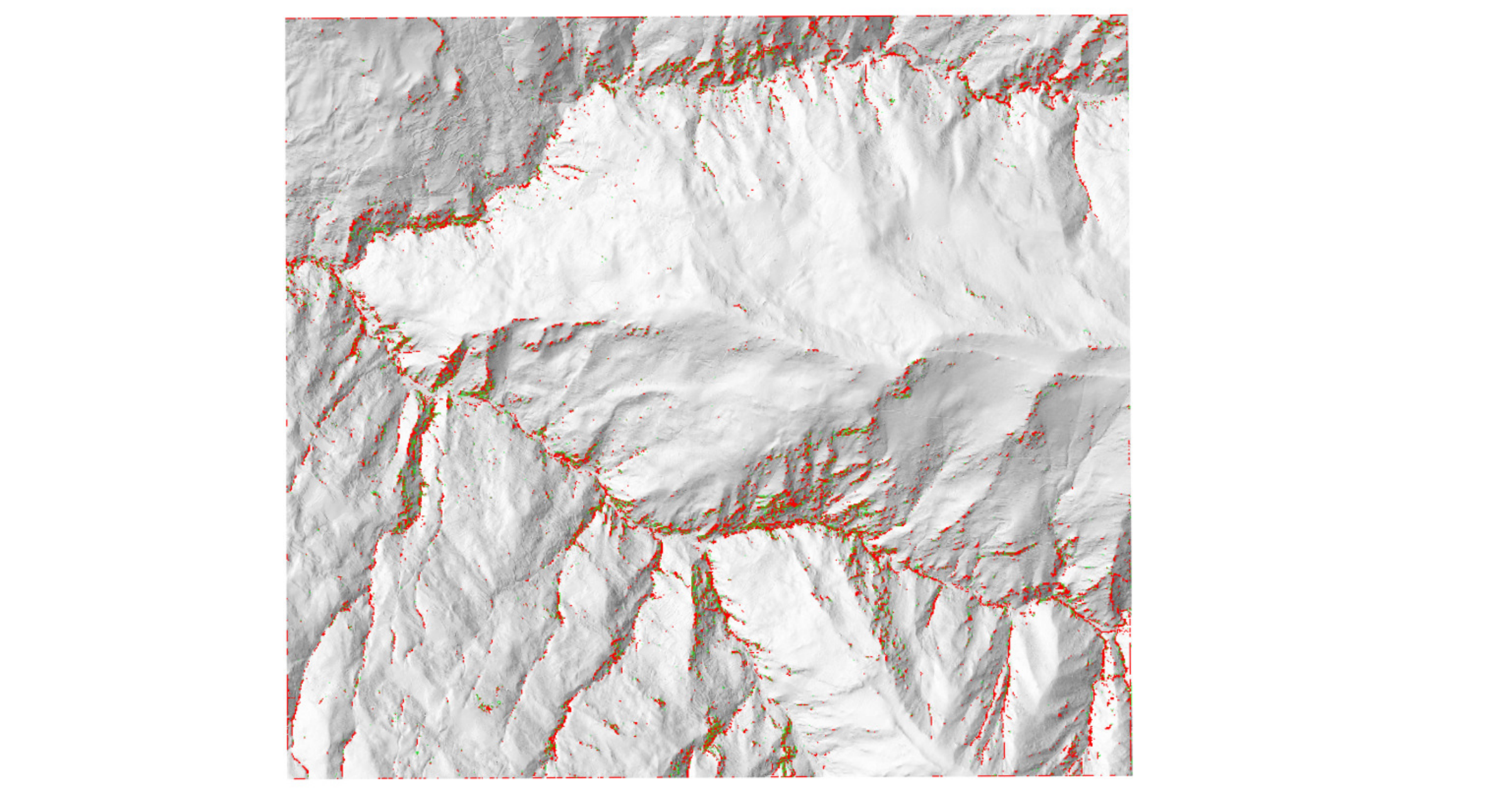} \hspace*{-1.5cm}
\includegraphics[scale=0.18]{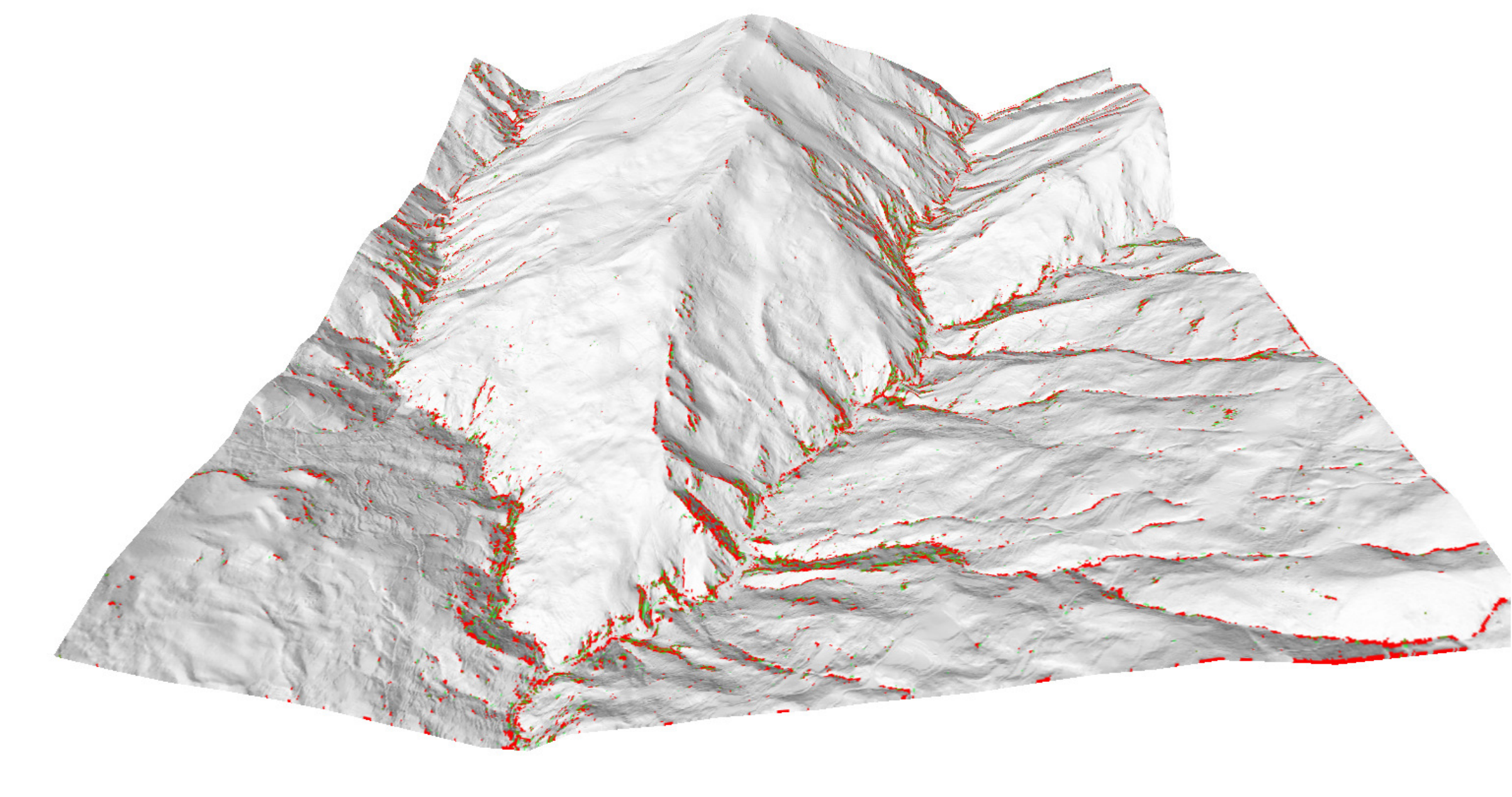}
\includegraphics[scale=0.15]{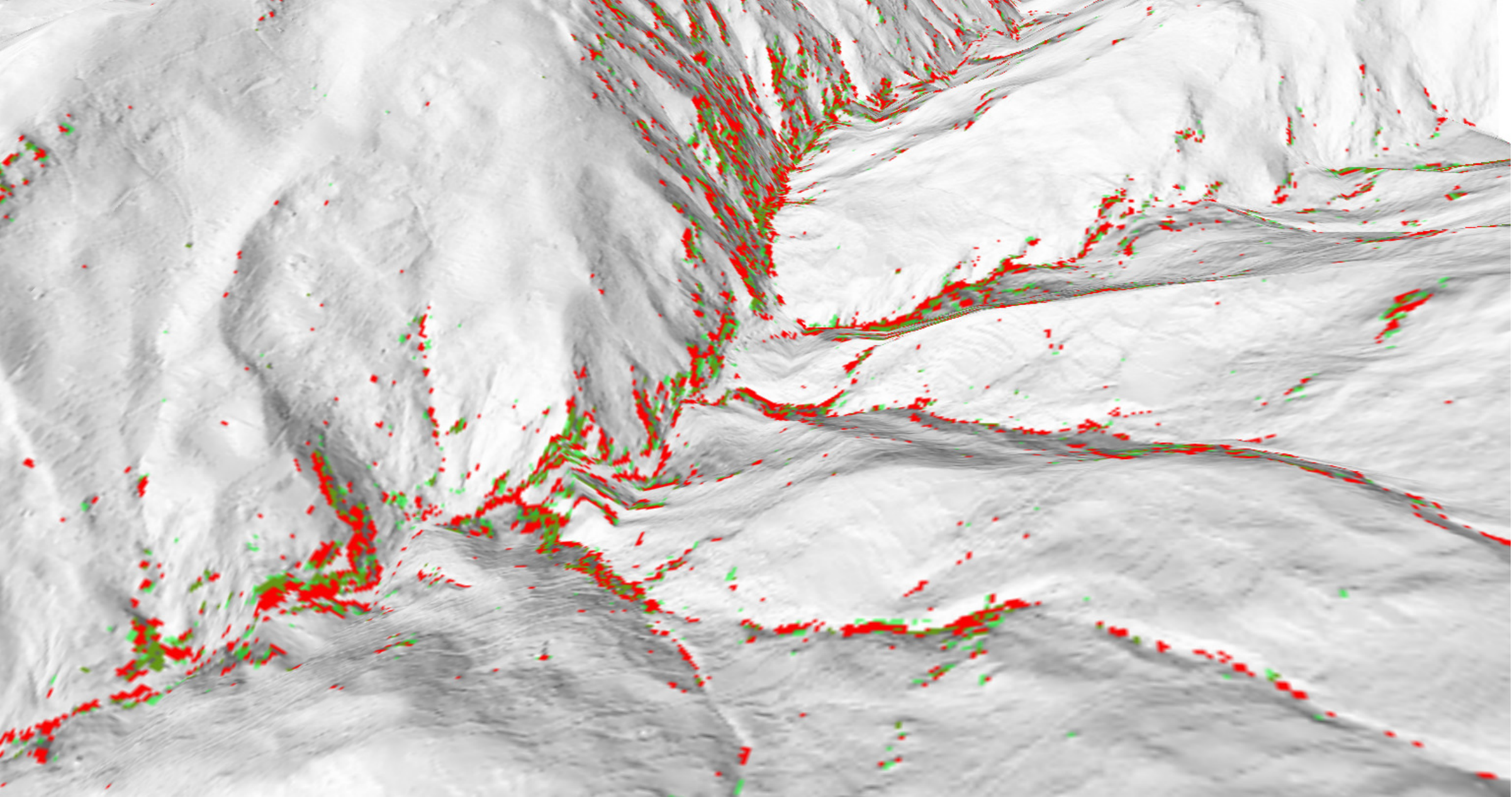}
\caption{Vertical accuracy measured as discrepancy comparing Shepard and cubic B-spline methods.}%
\label{manchas}%
\end{figure}

When using DTM$_{\text{dow\_Sh\_8m\_to\_2m}}$ the areas where the largest errors occur are very similar to those shown in Figure~\ref{manchas}, but a lower threshold has to be selected to get a similar amount of area covered by red and green areas.

Regarding the horizontal accuracy, Table~\ref{Tab2} shows the mean and standard deviation for both scenarios DTM$_{\text{up\_8m}}$ and DTM$_{\text{up\_16m}}$. Also in this case the performance of the multinode biquadratic Shepard method compared to the cubic B-spline implemented in QGIS is remarkable.

\begin{table}[tbp]
\centering%
\begin{tabular}
[c]{|c|c|c|c|c|}\hline
Horizontal & \multicolumn{2}{|c|}{Multinode Shepard} &
\multicolumn{2}{|c|}{{Cubic B-spline (QGIS)}}\\\cline{2-5}%
Accuracy & Mean ($\operatorname{m}$) & Sd ($\operatorname{m}$) & Mean
($\operatorname{m}$) & Sd ($\operatorname{m}$)\\\hline
MCN$_{\text{dow\_8m\_to\_2m}}$ & $0.227$ & $0.050$ & $0.521$ & $0.060$\\\hline
MCN$_{\text{dow\_16m\_to\_2m}}$ & $0.461$ & $0.096$ & $0.790$ & $0.129$%
\\\hline
\end{tabular}
\caption{Horizontal accuracy measured as discrepancy comparing the biquadratic
multinode Shepard method and the cubic B-spline method in QGIS}
\label{Tab2}
\end{table}

\section*{Acknowledgments}
The authors are grateful to the anonymous reviewers for carefully reading the
manuscript and for their precise and helpful suggestions which allowed to improve the work.
\begin{enumerate}
\item This research has been achieved as part of RITA \textquotedblleft Research
 ITalian network on Approximation'' and as part of the UMI group \enquote{Teoria dell'Approssimazione
 e Applicazioni}. The research was supported by GNCS-INdAM 2025 project \lq\lq Polinomi, Splines e Funzioni Kernel: dall’Approssimazione Numerica al Software Open-Source\rq\rq. The work of F. Nudo has been funded by the European Union – NextGenerationEU under the Italian National Recovery and Resilience Plan (PNRR), Mission 4, Component 2, Investment 1.2 \lq\lq Finanziamento di progetti presentati da giovani ricercatori\rq\rq,\ pursuant to MUR Decree No. 47/2025.

\item Project PID2022-139586NB-44 funded by
MCIN/AEI/10.13039/501100011033 and by European Union NextGenerationEU/PRTR.

\item Project QUAL21-011 (Modeling Nature) of the Consejería de Universidad, Investigación e Innovación of the Junta de Andalucía, Spain, and

\item The “María de Maeztu” Excellence Unit IMAG, reference CEX2020-001105- M, funded by \newline MCINN/AEI/10.13039/501100011033/ CEX2020-001105-M.
\end{enumerate}

\bibliographystyle{elsarticle-num}
\bibliography{references.bib}

\end{document}